\documentclass[a4paper,12pt,reqno]{amsart}
\usepackage{amsfonts}
\usepackage{amsmath}
\usepackage{amssymb}
\usepackage[a4paper]{geometry}
\usepackage{mathrsfs}
\usepackage{csquotes}
\usepackage[colorlinks]{hyperref}
\renewcommand\eqref[1]{(\ref{#1})} 
\numberwithin{equation}{section}
\theoremstyle{plain}
\newtheorem{thm}{Theorem}[section]
\newtheorem{prop}[thm]{Proposition}
\newtheorem{cor}[thm]{Corollary}
\newtheorem{lem}[thm]{Lemma}
 
\theoremstyle{definition}

\newtheorem{rem}[thm]{Remark}

\renewcommand{\wp}{\mathfrak S}
\newcommand{\Rn}{\mathbb R^{n}}
\def\G{{\mathbb G}}
\def\R{{\mathcal R}}
\def\E{{\mathbb E}}

\begin{document}

   \title[Hardy-Sobolev and GN inequalities]
   {Euler semigroup, Hardy-Sobolev and Gagliardo-Nirenberg type inequalities on homogeneous groups}

   \author[M. Ruzhansky]{Michael Ruzhansky}
\address{
  Michael Ruzhansky:
  \endgraf
  Department of Mathematics
  \endgraf
  Imperial College London
  \endgraf
  180 Queen's Gate, London SW7 2AZ
  \endgraf
  United Kingdom
  \endgraf
  {\it E-mail address} {\rm m.ruzhansky@imperial.ac.uk}
  }

 \author[D. Suragan]{Durvudkhan Suragan}
\address{
	Durvudkhan Suragan:
	\endgraf
	Department of Mathematics
	\endgraf
	School of Science and Technology, Nazarbayev University
	\endgraf
	53 Kabanbay Batyr Ave, Astana 010000
	\endgraf
	Kazakhstan
	\endgraf
	{\it E-mail address} {\rm durvudkhan.suragan@nu.edu.kz}
}
\author[N. Yessirkegenov]{Nurgissa Yessirkegenov}
\address{
  Nurgissa Yessirkegenov:
  \endgraf
  Institute of Mathematics and Mathematical Modelling
  \endgraf
  125 Pushkin str.
  \endgraf
  050010 Almaty
  \endgraf
  Kazakhstan
  \endgraf
  and
  \endgraf
  Department of Mathematics
  \endgraf
  Imperial College London
  \endgraf
  180 Queen's Gate, London SW7 2AZ
  \endgraf
  United Kingdom
  \endgraf
  {\it E-mail address} {\rm n.yessirkegenov15@imperial.ac.uk}
  }

\thanks{The first author was supported in parts by the EPSRC
 grant EP/R003025/1 and by the Leverhulme Grant RPG-2017-151. The third author was supported by the MESRK grant AP05133271. No new data was collected or generated during the course of research.}


     \begin{abstract} In this paper we describe the Euler semigroup $\{e^{-t\mathbb{E}^{*}\mathbb{E}}\}_{t>0}$ on homogeneous Lie
groups, which allows us to obtain various types of the Hardy-Sobolev and Gagliardo-Nirenberg type inequalities for the Euler operator $\E$. Moreover, the sharp remainder terms of the Sobolev type inequality, maximal Hardy inequality and $|\cdot|$-radial weighted Hardy-Sobolev type inequality are established.
      \end{abstract}
     \maketitle

     \tableofcontents
\section{Introduction}
In this paper we continue the research from \cite{RSY17_Euler} devoted to properties of Euler operators on homogeneous  groups, their consequences, and related analysis. In turn, this continues the research direction initiated in \cite{RS17a} devoted to Hardy and other functional inequalities in the setting of Folland and Stein's \cite{FS-Hardy} homogeneous groups.

Recall the following Sobolev type (improved Hardy) inequality: Let $1<p<\infty$. Then we have
\begin{equation}\label{Sob_intro}
\int_{\Rn}|f(x)|^{p} \leq \left(\frac{p}{n}\right)^{p}\int_{\Rn}|(x\cdot\nabla) f(x)|^{p}dx
\end{equation}
for all $f\in C_{0}^{\infty}(\Rn)$, where $\nabla$ is the standard gradient in $\mathbb{R}^{n}$.

In \cite[Proposition 1.4]{OS09} the authors showed that the Sobolev type inequality \eqref{Sob_intro} is equivalent to the following Hardy inequality when $p=2$ and $n\geq3$:
\begin{equation}\label{Hardy_intro}
\int_{\Rn}\frac{|g(x)|^{2}}{\|x\|^{2}} \leq \left(\frac{2}{n-2}\right)^{2}\int_{\Rn}\left|\frac{x}{\|x\|}\cdot \nabla g(x)\right|^{2}dx
\end{equation}
for all $f,g\in C_{0}^{\infty}(\Rn\backslash\{0\})$, where $\|x\|=\sqrt{x_{1}^{2}+...+x_{n}^{2}}$.

We note that the analysis of the remainder terms in such inequalities has a long history, for example, we refer to \cite{BL85}
for Sobolev inequalities, to \cite[Section 4]{BV97} and \cite{BM97} for Hardy inequalities and for many
others, and a more recent literature review to \cite{GM11}.

In this paper we are interested, among other things, in obtaining \eqref{Sob_intro} with remainder terms, that is, for example, in \cite[Corollary 4.4]{BEHL08} the authors obtained the following: Let $n\geq3$ and $0\leq \delta <n^{2}/4$. Then there exists a positive constant $C$ such that
\begin{equation}\label{Sob_remain_intro}
\|(x\cdot\nabla)f\|^{2}_{L^{2}(\Rn)}-\delta\|f\|_{L^{2}(\Rn)}^{2}\geq C\left(\frac{n^{2}}{4}-\delta\right)^{\frac{n-1}{n}}\||x|F(x)\|^{2}_{L^{2^{*}}(\Rn)}
\end{equation}
holds for all $f\in C_{0}^{\infty}(\Rn)$, where $F(x):=\mathcal{M}(f)=\frac{1}{\mathbb{S}^{n-1}}\int_{\mathbb{S}^{n-1}}f(r,y)dy$ is the integral mean of $f$ over the unit sphere $\mathbb{S}^{n-1}$, and $2^{*}=2n/(n-2)$, $r=|x|$. They derived \eqref{Sob_remain_intro} from the following Gagliardo-Nirenberg inequality \cite[Theorem 4.1]{BEHL08} using the explicit representation of the semi-group $e^{-t(x\cdot\nabla)^{*}(x\cdot\nabla)}$: Let $1\leq p<q<\infty$ and let $-i(\partial/\partial r) f\in L^{p}(\mathbb{R}\times \mathbb{S}^{n-1})$ and $f\in B^{p/(p-q)}(\Rn)$. Then there exists a positive constant $C=C(p,q)$ such that
\begin{equation} \label{GN01_intro}
\|f \|_{L^{q}(\mathbb{R}\times\wp)}\leq C \|\mathcal{R} f\|^{p/q}_{L^{p}(\mathbb{R}\times\wp)}\|f\|^{1-p/q}_{B^{p/(p-q)}(\Rn)},
\end{equation}
where $\mathcal{R}:=\frac{d}{d|x|}$ is the radial derivative, and $\|\cdot\|_{B^{\alpha}}$ is the Besov type norm associated with the semi-group.

We show analogues of \eqref{Sob_remain_intro} and \eqref{GN01_intro}, and calculate the semi-group $e^{-t\E^{*}\E}$ on homogeneous (Lie)
groups, where $\E$ is the Euler operator (see Section \ref{SEC:prelim} for more information). We can also obtain \eqref{Sob_remain_intro} with sharp constant without using \eqref{GN01_intro} and semi-group $e^{-t\E^{*}\E}$. We will also explain that the obtained homogeneous group results are not only analogues of the known Euclidean results, but also they give new inequalities even in Abelian cases with arbitrary quasi-norms (see Remark \ref{Stu_thm_rem}). Since we do not have sub-Laplacian and \enquote{horizontal} gradients on general homogeneous groups, we will work with radial derivative ($\R:=\frac{d}{dr}$) and Euler ($\E:=|x|\R$) operators. Some other discussions on this idea will be provided in the final section. We refer to \cite{RSY17_NoDEA} for horizontal versions of these inequalities. 

For the convenience of the reader we briefly summarise the obtained results. Let $\mathbb{G}$ be a homogeneous group of homogeneous dimension $Q$ and let us fix any homogeneous quasi-norm $|\cdot|$. Then we have the following results:
\begin{itemize}
\item Let $x=ry$ with $r=|x|$ and $|y|=1$. Then the semigroup $e^{-t\mathbb{E}^{*}\mathbb{E}}$ is given by
\begin{equation}\label{sem1_intro}
\begin{split}
(e^{-t\mathbb{E}^{*}\mathbb{E}}f)(x)&=\frac{e^{-tQ^{2}/4}}{\sqrt{4\pi t}}r^{-Q/2}\int^{\infty}_{0}e^{-\frac{(\ln r - \ln s )^{2}}{4t}}s^{-Q/2}f(sy)s^{Q-1}ds\\&=
\frac{e^{-tQ^{2}/4}}{\sqrt{4\pi t}}|x|^{-Q/2}\int^{\infty}_{0}e^{-\frac{(\ln |x| - \ln s )^{2}}{4t}}s^{-Q/2}f(sy)s^{Q-1}ds,
\end{split}
\end{equation}
where $\mathbb{E}=|x|\R$ is the Euler operator and $\mathcal{R}:=\frac{d}{d|x|}$ is the radial derivative.
\item Let $1\leq p<q<\infty$ and let $f$ be such that $\R f\in L^{p}(\mathbb{R}\times\wp)$ and $f\in B^{p/(p-q)}(\mathbb{R}\times\wp)$. Then there exists a positive constant $C=C(p,q)$ such that
\begin{equation} \label{GN1_intro}
\|f \|_{L^{q}(\mathbb{R}\times\wp)}\leq C \|\R f\|^{p/q}_{L^{p}(\mathbb{R}\times\wp)}\|f\|^{1-p/q}_{B^{p/(p-q)}(\mathbb{R}\times\wp)},
\end{equation}
where $\mathcal{R}:=\frac{d}{d|x|}$ is the radial derivative, $\wp$ is the unit sphere in $\G$ and $\|\cdot\|_{B^{\alpha}}$ is the Besov type norm defined by \eqref{Bes_norm}.
\item Let $Q\geq3$. Then we have for all $f\in C^{\infty}_{0}(\G)$ and $0\leq \delta <\frac{Q^{2}}{4}$ the inequality
\begin{equation}\label{Stu1_intro}
\begin{split}
 \int_{\G}|\E f(x)|^{2}dx&-\delta\int_{\G} |f(x)|^{2}dx\\& \geq\frac{\left(\frac{Q^{2}}{4}-\delta\right)^{\frac{Q-1}{Q}}}{\left(\frac{(Q-2)^{2}}{4}\right)^{\frac{Q-1}{Q}}}S_{Q}
\left(\int_{\G}|x|^{2^{*}}|g(|x|)|^{2^{*}}dx\right)^{\frac{2}{2^{*}}}
\end{split}
 \end{equation}
with sharp constant, where $g(|x|)=\mathcal{M}(f)(|x|):=\frac{1}{|\wp|}\int_{\wp}f(|x|,y)d\sigma(y)$, $2^{*}=2Q/(Q-2)$ and $S_{Q}$ is the constant defined in \eqref{S_Q}.
\item Let $\phi$ and $\psi$ be positive functions defined on $\G$. Then there exists a positive constant $C$ such that
\begin{equation}\label{high2_wei_Hardy1_intro}
\int_{\G}\phi(x)\exp (\mathcal{M} \log f)(x)dx\leq C\int_{\G} \psi(x)f(x)dx
\end{equation}
holds for all positive $f$ if and only if $A(\phi,\psi)<\infty$, where $A$ is given in \eqref{high2_wei_Hardy2} and $(\mathcal{M}f)(x)=\frac{1}{|B(0,|x|)|}\int_{B(0,|x|)}f(z)dz$.
\item Let $Q\geq2$. For any quasi-norm $|\cdot|$, all differentiable $|\cdot|$-radial functions $\phi$, all $p>1$ and all $f \in C_{0}^{1}(\mathbb{G})$ we have
\begin{align}\label{4}
\int_{\mathbb{G}} \frac{\phi^{\prime}(|x|)}{|x|^{Q-1}}|f(x)|^{p}dx \leq \int_{\mathbb{G}} \left| \R f(x) \right|^{p} dx  + (p-1)\int_{\mathbb{G}} \frac{|\phi(|x|)|^{\frac{p}{p-1}}}{|x|^{\frac{p(Q-1)}{p-1}}}  |f(x)|^{p}dx. \nonumber
\end{align}	
In the Abelian case with the standard Euclidean distance $\|x\|=\sqrt{x^{2}_{1}+\ldots+x^{2}_{n}}$ it implies the \enquote{critical} Hardy inequality
\begin{equation}
\int_{\mathbb R^{n}}
\frac{|f(x)|^{n}}{\|x\|^{n}} dx \leq \int_{\mathbb R^{n}} \left| \nabla f(x) \right|^{n} dx  + (n-1)\int_{\mathbb R^{n}} \frac{\left|\log \frac{1}{\|x\|}\right|^{\frac{n}{n-1}}}{\|x\|^{n}} \left| f(x) \right|^{n} dx,
\end{equation}
where $\nabla $ is the standard gradient on $\mathbb R^{n}$.
\item Let $Q\geq2$. For any quasi-norm $|\cdot|$, all differentiable $|\cdot|$-radial functions $\phi$, all $p>1$ and all $f \in C_{0}^{1}(\mathbb{G})$ we have
\begin{align}
	\int_{\mathbb{G}} \frac{\phi^{\prime}(|x|)}{|x|^{Q-1}}|f(x)|^{p}dx \leq p \left(\int_{\mathbb{G}} \left| \R f(x) \right|^{p} dx \right)^{\frac{1}{p}}  \left(\int_{\mathbb{G}} \frac{|\phi(|x|)|^{\frac{p}{p-1}}}{|x|^{\frac{p(Q-1)}{p-1}}}  |f(x)|^{p}dx\right)^{\frac{p-1}{p}}. \nonumber
\end{align}	
In the Abelian case with the standard Euclidean distance $\|x\|=\sqrt{x^{2}_{1}+\ldots+x^{2}_{n}}$ it implies the \enquote{improved} Heisenberg-Pauli-Weyl uncertainty principle
\begin{equation}\label{introHPWp=2}
\left(\int_{\mathbb R^{n}}|f(x)|^{2}dx\right)^{2} \leq  \left(\frac{2}{n}\right)^{2}  \int_{\mathbb R^{n}} \left| \nabla f(x) \right|^{2} dx  \int_{\mathbb R^{n}} \|x\|^{2} |f(x)|^{2}dx, \quad n\geq2,
\end{equation}
where $\nabla $ is the standard gradient on $\mathbb R^{n}$.

Indeed the same inequality with the constant $\left(\frac{2}{n-2}\right)^{2}, n\geq3,$ (instead of $\left(\frac{2}{n}\right)^{2}$) is known as the Heisenberg-Pauli-Weyl uncertainty principle (see, e.g. \cite[Remark 2.10]{Ruzhansky-Suragan:uncertainty}), that is,
\begin{equation}\label{introRS17}
	\left(\int_{\mathbb R^{n}}|f|^{2}dx\right)^{2} \leq  \left(\frac{2}{n-2}\right)^{2}  \int_{\mathbb R^{n}} \left| \nabla f \right|^{2} dx  \int_{\mathbb R^{n}} \|x\|^{2} |f|^{2}dx, \quad n\geq3,
\end{equation}
for all $f\in C^{1}_{0}(\mathbb R^{n}).$
Obviously, since $\frac{2}{n-2}\geq \frac{2}{n},\, n\geq3,$ inequality \eqref{introHPWp=2}  is an improved version of \eqref{introRS17}. Note that equality case in \eqref{introHPWp=2} holds for the function $f=C\exp (-b \|x\|),\,b>0.$
\end{itemize}

The organisation of the paper is as follows. In Section \ref{SEC:prelim} we briefly recall the necessary concepts of homogeneous Lie
groups and fix the notation. The operator semigroup $\{e^{-t\mathbb{E}^{*}\mathbb{E}}\}_{t>0}$ is determined in Section \ref{SEC:semigroup}. In Section \ref{SEC:GN} we establish Gagliardo-Nirenberg type inequalities, and obtain sharp remainder terms of the Sobolev type inequality. The maximal Hardy inequality is established in Section \ref{SEC:Max}. In Section \ref{final} we give some further discussions on (critical) Hardy-Sobolev type inequalities on homogeneous groups.

\section{Preliminaries}
\label{SEC:prelim}

In this section we briefly recall the necessary notions and fix the notation for homogeneous groups. We refer to \cite{FR} for a detailed discussion of the appearing objects.

Let us recall the notion of a family of dilations of a Lie algebra $\mathfrak{g}$, that is, a family of linear mappings of the following form
$$D_{\lambda}={\rm Exp}(A \,{\rm ln}\lambda)=\sum_{k=0}^{\infty}
\frac{1}{k!}({\rm ln}(\lambda) A)^{k},$$
where $A$ is a diagonalisable linear operator on $\mathfrak{g}$ with positive eigenvalues. We also recall that $D_{\lambda}$ is a morphism of $\mathfrak{g}$, if it is a linear mapping from $\mathfrak{g}$ to itself satisfying the property
$$\forall X,Y\in \mathfrak{g},\, \lambda>0,\;
[D_{\lambda}X, D_{\lambda}Y]=D_{\lambda}[X,Y],$$
where $[X,Y]:=XY-YX$ is the Lie bracket. Then, a {\em homogeneous group} $\G$ is a connected simply connected Lie group whose Lie algebra is equipped with a morphism family of dilations. It induces the dilation structure on $\mathbb G$ which we denote by $D_{\lambda}x$ or just by $\lambda x$.

Let $dx$ be the Haar measure on $\mathbb{G}$ and let $|S|$ denote the volume of a measurable subset $S$ of $\G$. The homogeneous dimension of $\mathbb G$ is defined by
$$Q := {\rm Tr}\,A.$$
Then we have
\begin{equation}
|D_{\lambda}(S)|=\lambda^{Q}|S| \quad {\rm and}\quad \int_{\mathbb{G}}f(\lambda x)
dx=\lambda^{-Q}\int_{\mathbb{G}}f(x)dx.
\end{equation}
A {\em homogeneous quasi-norm} on $\mathbb G$ is
a continuous non-negative function
$$\mathbb{G}\ni x\mapsto |x|\in [0,\infty)$$
satisfying the properties
\begin{itemize}
\item   $|x^{-1}| = |x|$ for all $x\in \mathbb{G}$,
\item  $|\lambda x|=\lambda |x|$ for all
$x\in \mathbb{G}$ and $\lambda >0$,
\item  $|x|= 0$ if and only if $x=0$.
\end{itemize}

The quasi-ball centred at $x\in\mathbb{G}$ with radius $R > 0$ can be defined by
$$B(x,R):=\{y\in \mathbb{G}: |x^{-1}y|<R\}.$$
The polar decomposition on homogeneous groups can be formulated as follows: there is a (unique)
positive Borel measure $\sigma$ on the unit sphere
\begin{equation}\label{EQ:sphere}
\wp:=\{x\in \mathbb{G}:\,|x|=1\},
\end{equation}
such that
\begin{equation}\label{EQ:polar}
\int_{\mathbb{G}}f(x)dx=\int_{0}^{\infty}
\int_{\wp}f(ry)r^{Q-1}d\sigma(y)dr
\end{equation}
holds for all $f\in L^{1}(\mathbb{G})$.

We refer to Folland and Stein \cite{FS-Hardy} for more information (see also \cite[Section 3.1.7]{FR} for a detailed discussion).

If we fix a basis $\{X_{1},\ldots,X_{n}\}$ of $\mathfrak{g}$
such that we have
$$AX_{k}=\nu_{k}X_{k}$$
for each $k$, then the matrix $A$ can be taken to be
$A={\rm diag} (\nu_{1},\ldots,\nu_{n})$, where each $X_{k}$ is homogeneous of degree $\nu_{k}$ and
$$
Q=\sum_{k=1}^{n}\nu_{k}.
$$
The decomposition of ${\exp}_{\mathbb{G}}^{-1}(x)$ in $\mathfrak g$ defines the vector
$$e(x)=(e_{1}(x),\ldots,e_{n}(x))$$
by the formula
$${\exp}_{\mathbb{G}}^{-1}(x)=e(x)\cdot \nabla\equiv\sum_{j=1}^{n}e_{j}(x)X_{j},$$
where $\nabla=(X_{1},\ldots,X_{n})$.
On the other hand, we also have the following equality
$$x={\exp}_{\mathbb{G}}\left(e_{1}(x)X_{1}+\ldots+e_{n}(x)X_{n}\right).$$
Using homogeneity and denoting $x=ry,\,y\in \wp,$ we obtain
$$
e(x)=e(ry)=(r^{\nu_{1}}e_{1}(y),\ldots,r^{\nu_{n}}e_{n}(y)).
$$
By a direct calculation, we get
\begin{equation*}
\frac{d}{d|x|}f(x)=\frac{d}{dr}f(ry)=
 \frac{d}{dr}f({\exp}_{\mathbb{G}}
\left(r^{\nu_{1}}e_{1}(y)X_{1}+\ldots
+r^{\nu_{n}}e_{n}(y)X_{n}\right)),
\end{equation*}
which implies the following equality
\begin{equation}\label{dfdr}
	\frac{d}{d|x|}f(x)=\mathcal{R}f(x),\;\;\forall x\in \G,
\end{equation}
for each homogeneous quasi-norm $|x|$ on a homogeneous group $\mathbb G$, where we have denoted
\begin{equation}\label{EQ:Euler}
\mathcal{R} :=\frac{d}{dr}.
\end{equation}
We will also use the Euler operator
\begin{equation}\label{Euler_proper}
\mathbb{E}:=|x|\R.
\end{equation}
Let us recall the following property of the Euler operator $\mathbb{E}$.
\begin{lem}[{\cite[Lemma 2.2]{RSY17_Euler}}]\label{L2:E} We have
\begin{equation}\label{L2:E1}
\mathbb{E^{*}}=-Q\mathbb{I}-\mathbb{E},
\end{equation}
where $\mathbb{I}$ and $\mathbb{E^{*}}$ are the identity operator and the formal adjoint operator of $\mathbb{E}$, respectively.
\end{lem}
Now we give useful inequality and identity for Euler operator.
\begin{lem}\label{Euler_lem_proper2} Let $1<p<\infty$ and let $\{\phi_{i}\}_{i=1}^{3}\in L_{loc}^{1}(\G)$ be any radial functions. If we denote
\begin{equation}\label{Bli8}
\widetilde{f}(r):=\left(\frac{1}{|\wp|}\int_{\wp}|f(ry)|^{p}d\sigma(y)\right)^{1/p},
\end{equation}
then we have
\begin{equation}\label{Euler_lem_proper2_1}
\int_{\mathbb{G}}\phi_{1}(x)\left|\widetilde{f}(|x|)\right|^{p}dx
=\int_{\mathbb{G}}\phi_{1}(x)\left|f(x)\right|^{p}dx
\end{equation}
for any $f\in L_{loc}^{p}(\G)$.
Moreover, we have
\begin{equation}\label{Bli9}
\int_{\mathbb{G}}\phi_{2}(x)\left|\E^{k} \widetilde{f}(|x|)\right|^{p}dx
\leq\int_{\mathbb{G}}\phi_{2}(x)\left|\E^{k} f(x)\right|^{p}dx
\end{equation}
and
\begin{equation}\label{Euler_lem_proper2_2}
\int_{\mathbb{G}}\phi_{3}(x)\left|\R^{k} \widetilde{f}(|x|)\right|^{p}dx
\leq\int_{\mathbb{G}}\phi_{3}(x)\left|\R^{k} f(x)\right|^{p}dx
\end{equation}
for any $k\in \mathbb{N}$ and all $f\in L_{loc}^{p}(\G)$ such that $\E^{k} f\in L_{loc}^{p}(\G)$ with sharp constants. The constants are attained when $f=\widetilde{f}$.
\end{lem}
\begin{rem} In most cases, this lemma allows that instead of proving an inequality for non-radial functions, it is enough to prove it for radial functions (see e.g. in the proof of Theorem \ref{Stu_thm} and \cite[in the proof of Theorem 3.1]{RSY17_Tran}).
\end{rem}
\begin{proof}[Proof of Lemma \ref{Euler_lem_proper2}] Using representation \eqref{Bli8}, we obtain
$$\int_{\mathbb{G}}|\widetilde{f}(|x|)|^{p}\phi_{1}(x)dx=|\wp|\int_{0}^{\infty}|\widetilde{f}(r)|^{p}\phi_{1}(r)r^{Q-1}dr$$
\begin{equation}\label{aremterm6}=|\wp|\int_{0}^{\infty}\frac{1}{|\wp|}\int_{\wp}|f(ry)|^{p}d\sigma(y)\phi_{1}(r)r^{Q-1}dr=
\int_{\mathbb{G}}|f(x)|^{p}\phi_{1}(x)dx, \end{equation}
which is \eqref{Euler_lem_proper2_1}.

Now we prove \eqref{Bli9}. First let us prove the following by induction
\begin{equation}\label{rad_non_rad_rel}
|\E^{k}\widetilde{f}|\leq \left(\frac{1}{|\wp|}\int_{\wp}\left|\E^{k} f(ry)\right|^{p}d\sigma(y)\right)^{\frac{1}{p}}
\end{equation}
for any $k\in \mathbb{N}$. So, we need to check this for $k=1$. Using H\"{o}lder's inequality, we get
$$|\E\widetilde{f}|=\left|r\frac{d}{dr}\widetilde{f}(r)\right|=r\left(\frac{1}{|\wp|}\int_{\wp}|f(ry)|^{p}d\sigma(y)\right)^{\frac{1}{p}-1}
\frac{1}{|\wp|}\left|\int_{\wp}|f(ry)|^{p-2}f(ry)\overline{\frac{d}{dr}f(ry)}d\sigma(y)\right|$$
$$\leq\left(\frac{1}{|\wp|}\int_{\wp}|f(ry)|^{p}d\sigma(y)\right)^{\frac{1}{p}-1}
\frac{1}{|\wp|}\int_{\wp}|f(ry)|^{p-1}\left|r\frac{d}{dr}f(ry)\right|d\sigma(y)$$
$$\leq \left(\frac{1}{|\wp|}\int_{\wp}|f(ry)|^{p}d\sigma(y)\right)^{\frac{1}{p}-1}
\frac{1}{|\wp|}\left(\int_{\wp}\left|r\frac{d}{dr}f(ry)\right|^{p}d\sigma(y)\right)^{\frac{1}{p}}
\left(\int_{\wp}|f(ry)|^{p}d\sigma(y)\right)^{\frac{p-1}{p}}$$
$$=\left(\frac{1}{|\wp|}\int_{\wp}\left|\E f(ry)\right|^{p}d\sigma(y)\right)^{\frac{1}{p}}.$$
Now assuming that
\begin{equation}\label{k_ell}
|\E^{\ell}\widetilde{f}|\leq \left(\frac{1}{|\wp|}\int_{\wp}\left|\E^{\ell} f(ry)\right|^{p}d\sigma(y)\right)^{\frac{1}{p}}
\end{equation}
holds for $\ell\in \mathbb{N}$, we prove
$$
|\E^{\ell+1}\widetilde{f}|\leq \left(\frac{1}{|\wp|}\int_{\wp}\left|\E^{\ell+1} f(ry)\right|^{p}d\sigma(y)\right)^{\frac{1}{p}}.
$$
Then, using \eqref{k_ell} we calculate
\begin{multline*}
|\E^{\ell+1} \widetilde{f}(r)|=|\E(\E^{\ell} \widetilde{f}(r))|
\leq \left|\E\left(\left(\frac{1}{|\wp|}\int_{\wp}\left|\E^{\ell} f(ry)\right|^{p}d\sigma(y)\right)^{\frac{1}{p}}\right)\right|\\
=r\left(\frac{1}{|\wp|}\int_{\wp}\left|\E^{\ell} f(ry)\right|^{p}d\sigma(y)\right)^{\frac{1}{p}-1}
\left|\frac{1}{|\wp|}\int_{\wp}\left|\E^{\ell} f(ry)\right|^{p-2}\E f(ry) \overline{\frac{d}{dr}(\E^{\ell}f(ry))}d\sigma(y)\right|\\
\leq \left(\frac{1}{|\wp|}\int_{\wp}\left|\E^{\ell} f(ry)\right|^{p}d\sigma(y)\right)^{\frac{1}{p}-1}
\left(\frac{1}{|\wp|}\int_{\wp}\left|\E^{\ell} f(ry)\right|^{p-1}\left|\E^{\ell+1}f(ry)\right|d\sigma(y)\right)\\
\leq \left(\frac{1}{|\wp|}\int_{\wp}|\E^{\ell}f(ry)|^{p}d\sigma(y)\right)^{\frac{1}{p}-1}
\frac{1}{|\wp|}\left(\int_{\wp}\left|\E^{\ell+1} f(ry)\right|^{p}d\sigma(y)\right)^{\frac{1}{p}}
\\ \times\left(\int_{\wp}|\E^{\ell} f(ry)|^{p}d\sigma(y)\right)^{\frac{p-1}{p}}
=\left(\frac{1}{|\wp|}\int_{\wp}\left|\E^{\ell+1} f(ry)\right|^{p}d\sigma(y)\right)^{\frac{1}{p}},
\end{multline*}
where we have used H\"{o}lder's inequality in the last line. Thus, we have proved \eqref{rad_non_rad_rel}.

Now using \eqref{rad_non_rad_rel}, we obtain
\begin{equation*}
\begin{split}
\int_{\mathbb{G}}\left|\E^{k} \widetilde{f}(x)\right|^{p}\phi_{2}(x)dx&=|\wp|\int_{0}^{\infty}\left|\E^{k} \widetilde{f}(r)\right|^{p}\phi_{2}(r)r^{Q-1}dr\\&\leq
|\wp|\int_{0}^{\infty}\frac{1}{|\wp|}\int_{\wp}\left|\E^{k} f(ry)\right|^{p}\phi_{2}(r)r^{Q-1}d\sigma(y)dr\\&
=\int_{\mathbb{G}}\left|\E^{k} f(x)\right|^{p}\phi_{2}(x)dx.
\end{split}
\end{equation*}
Thus, we have proved \eqref{Bli9} for any $k\in\mathbb{N}$. Similarly, one can obtain \eqref{Euler_lem_proper2_2}.
\end{proof}
\section{Euler semigroup $e^{-t\mathbb{E}^{*}\mathbb{E}}$}
\label{SEC:semigroup}
In this section we introduce the operator semigroup $\{e^{-t\mathbb{E}^{*}\mathbb{E}}\}_{t>0}$ associated with the Euler operator on homogeneous groups.

First, let us prepare some preliminary results. We define the map $F : L^{2}(\G)\rightarrow L^{2}(\mathbb{R}\times \wp)$ as
\begin{equation}\label{sem2}
(Ff)(s,y) :=e^{sQ/2}f(e^{s}y)
\end{equation}
for $y \in \wp$ and $s\in \mathbb{R}$, and its inverse $F^{-1}: L^{2}(\mathbb{R}\times \wp)\rightarrow L^{2}(\G)$ as
\begin{equation}\label{sem3}
(F^{-1}g)(x) :=r^{-Q/2}g(\ln r, y).
 \end{equation}
We note that $F$ preserves the $L^{2}$ norm.

We will also use the dilations $U(t):L^{2}(\G)\rightarrow L^{2}(\G)$ given by
\begin{equation}\label{sem4}
U(t)f(x):=e^{tQ/2}f(e^{t}x).
 \end{equation}
These form a group of unitary operators with generator $U(t)=e^{iAt}$, where $A$ is given by
\begin{equation}\label{sem_A1}
Af=\frac{1}{i}\frac{d}{dt}U(t)f\lvert_{t=0}=\frac{1}{i}\left(\E+\frac{Q}{2}\right)f=-i\E f-i\frac{Q}{2}f,
\end{equation}
where we have used that
$$\frac{d}{dt}f(e^{t}x)=\E f(e^{t}x),$$
as follows from the following equality with $x=ry$ and $\rho:=e^{t}r$,
$$\frac{d}{dt}(f(e^{t}x))=\frac{d}{dt}(f(e^{t}ry))=\frac{d}{dt}(f(\rho y))$$$$=\rho\frac{d}{d\rho}(f(\rho y))=\E f(\rho y)=\E f(e^{t}x),$$
where $y\in \wp$.
Using Lemma \ref{L2:E} we obtain from \eqref{sem_A1} that
\begin{equation}\label{sem_A2}
A=A^{*}=-i\E-i\frac{Q}{2},
\end{equation}
which implies
\begin{equation}\label{sem_A3}
\E^{*}\E=\left(-iA-\frac{Q}{2}\right)\left(iA-\frac{Q}{2}\right)=A^{2}+\frac{Q^{2}}{4}.
\end{equation}
By \eqref{sem2} and \eqref{sem4}, we have $(Ff)(s,y)=(U(s)f)(y)$ for $y\in\wp$ and $s\in\mathbb{R}$, then this with the group property of the dilations $U(\cdot)$ gives that
\begin{equation}\label{dil_shift}
(F(U(t)f))(s,y)=(U(s)(U(t)f))(y)=(U(s+t)f)(y)=(F f)(s+t,y).
\end{equation}
If we define $M :L^{2}(\G)\rightarrow L^{2}(\mathbb{R}\times \wp)$ as
\begin{equation}\label{sem5}
(Mf)(\tau , y) :=\frac{1}{\sqrt{2\pi}}\int_{\mathbb{R}}e^{-is\tau}(Ff)(s,y)ds,
\end{equation}
then it follows from \eqref{dil_shift} and the change of variables that
\begin{equation}\label{sem6}
\begin{split}
(MU(t)f)(\tau, y)&=\frac{1}{\sqrt{2\pi}}\int_{\mathbb{R}} e^{-is\tau}(Ff)(s+t,y)ds\\&=
\frac{e^{it\tau}}{\sqrt{2\pi}}\int_{\mathbb{R}} e^{-is\tau}(Ff)(s,y)ds\\&=e^{it\tau}(Mf)(\tau,y).
\end{split}
\end{equation}

We note that the map $M={\mathcal{F}}\circ F$ is the Mellin transformation, where $\mathcal{F}$ is the Fourier transform on $\mathbb{R}$. The map $M$ has an explicit representation using the group structure of $\mathbb{R}^{+}$ under multiplication: it is the Fourier transform on this group.

Now we are ready to give the representation of the operator semigroup $\{e^{-t\mathbb{E}^{*}\mathbb{E}}\}_{t>0}$ on homogeneous groups.
\begin{thm}\label{sem_thm} Let $\mathbb{G}$ be a homogeneous group
of homogeneous dimension $Q$. Let $x=ry$ and $r=|x|$ with $y\in \wp$. Then the semigroup $e^{-t\mathbb{E}^{*}\mathbb{E}}$ is given by
\begin{equation}\label{sem1}
\begin{split}
(e^{-t\mathbb{E}^{*}\mathbb{E}}f)(x)&=\frac{e^{-tQ^{2}/4}}{\sqrt{4\pi t}}r^{-Q/2}\int^{\infty}_{0}e^{-\frac{(\ln r - \ln s )^{2}}{4t}}s^{-Q/2}f(sy)s^{Q-1}ds\\&=
\frac{e^{-tQ^{2}/4}}{\sqrt{4\pi t}}|x|^{-Q/2}\int^{\infty}_{0}e^{-\frac{(\ln |x| - \ln s )^{2}}{4t}}s^{-Q/2}f(sy)s^{Q-1}ds.
\end{split}
\end{equation}
\end{thm}
\begin{proof}[Proof of Theorem \ref{sem_thm}] Let us first show that
\begin{equation}\label{sem8_0}
(MAf)(\tau,y)=\tau(Mf)(\tau,y)
\end{equation}
for $f$ in the domain $\mathcal{D}(A)$, which implies the fact that $f\in \mathcal{D}(A)\Leftrightarrow (\tau,y)\mapsto \tau(Mf)(\tau,y)\in L^{2}(\mathbb{R}\times \wp)$. Noting $iAe^{itA}=\partial_{t}U(t)$ we calculate
$$(M iAe^{iAt}f)(\tau,y)=(M \partial_{t}U(t)f)(\tau,y)=\partial_{t}(MU(t)f)(\tau,y).$$
Now using \eqref{sem6} we get from above that
\begin{equation}\label{sem8_1}
(M iAe^{iAt}f)(\tau,y)=\partial_{t}e^{it\tau}(Mf)(\tau,y)=i\tau e^{it\tau}(Mf)(\tau,y),
\end{equation}
which implies \eqref{sem8_0} after setting $t=0$.

Now we prove that
\begin{equation}\label{sem8_2}
(Me^{-tA^{2}}f)(\tau,y)=e^{-t\tau^{2}}(Mf)(\tau, y).
\end{equation}
To obtain \eqref{sem8_2}, we write
\begin{equation}\label{sem8_3}
(Me^{-tA^{2}}f)(\tau,y)=\sum_{k=0}^{\infty}\frac{(-t)^{k}}{k!}(MA^{2k}f)(\tau,y).
\end{equation}
On the other hand, by iteration we have from \eqref{sem8_0} that
$$(MA^{2k}f)(\tau,y)=\tau^{2k}(Mf)(\tau,y),\;\;\;k=0,1,2\ldots.$$
Putting this in \eqref{sem8_3}, we obtain
$$(Me^{-tA^{2}}f)(\tau,y)=\sum_{k=0}^{\infty}\frac{(-t)^{k}}{k!}\tau^{2k}(Mf)(\tau,y)=e^{-t\tau^{2}}(Mf)(\tau, y).$$
Thus, we have obtained \eqref{sem8_2}. Then, it implies that
$$e^{-tA^{2}}=M^{-1}e^{-t\tau^{2}}M.$$
Here using $M=\mathcal{F}\circ F$, one has
\begin{equation}\label{sem8}
e^{-tA^{2}}= F^{-1}\circ\mathcal{F}^{-1}(e^{-t\tau^{2}}\mathcal{F}\circ F).
\end{equation}
Then, a direct calculation gives that
\begin{equation*}
\begin{split}
\mathcal{F}^{-1}(e^{-t\tau^{2}}Mf)(\lambda,y)&=\mathcal{F}^{-1}(e^{-t\tau^{2}}\mathcal{F}\circ F)(\lambda,y)
\\&=\frac{1}{2\pi}\int_{\mathbb{R}}\int_{\mathbb{R}}e^{i\lambda\tau}e^{-t\tau^{2}}e^{-is\tau}(Ff)(s,y)ds d\tau
\\&=\frac{1}{2\pi}\int_{\mathbb{R}}\left(\int_{\mathbb{R}}e^{-t\tau^{2}+i(\lambda-s)\tau} d\tau\right)(Ff)(s,y)ds\\&
=\frac{1}{\sqrt{4\pi t}}\int_{\mathbb{R}} e^{-\frac{(\lambda-s)^{2}}{4t}}(Ff)(s,y)ds=:\varphi_{t}(\lambda,y),
\end{split}
\end{equation*}
where we have used the property of the Gaussian integral in the last line
$$\int_{\mathbb{R}}e^{-t\tau^{2}+i(\lambda-s)\tau} d\tau=\sqrt{\frac{\pi}{t}}e^{-\frac{(\lambda-s)^{2}}{4t}}.$$
From this and \eqref{sem8}, using \eqref{sem2}, \eqref{sem3} and $M=\mathcal{F}\circ F$ with $x=ry$ we compute
\begin{equation*}
\begin{split}
(e^{-tA^{2}}f)(ry)&=(F^{-1}\varphi_{t})(ry)\\&
=r^{-Q/2}\varphi_{t}(\ln r,y)\\&
=\frac{1}{\sqrt{4\pi t}}r^{-Q/2}\int_{\mathbb{R}}e^{-\frac{(\ln r-s)^{2}}{4t}}(Ff)(s,y)ds\\&
=\frac{1}{\sqrt{4\pi t}}r^{-Q/2}\int_{0}^{\infty}e^{-\frac{(\ln r-\ln z)^{2}}{4t}}z^{\frac{Q}{2}-1}f(zy)dz,
\end{split}
\end{equation*}
where we have used the change of variables $z=e^{s}$ in the last line.

Since we have $e^{-t\E^{*}\E}=e^{-tQ^{2}/4}e^{-tA^{2}}$ by \eqref{sem_A3}, we obtain from above that
\begin{equation*}
\begin{split}
(e^{-t\E^{*}\E}f)(ry)&=e^{-tQ^{2}/4}(e^{-tA^{2}}f)(ry)\\&
=\frac{1}{\sqrt{4\pi t}}r^{-Q/2}e^{-tQ^{2}/4}\int_{0}^{\infty}e^{-\frac{(\ln r-\ln z)^{2}}{4t}}z^{\frac{Q}{2}-1}f(zy)dz\\&
=\frac{1}{\sqrt{4\pi t}}r^{-Q/2}e^{-tQ^{2}/4}\int_{0}^{\infty}e^{-\frac{(\ln r-\ln z)^{2}}{4t}}z^{-\frac{Q}{2}}f(zy)z^{Q-1}dz,
\end{split}
\end{equation*}
yielding \eqref{sem1}.
\end{proof}
Now let us give the following representation for $e^{-tA^{2}}$, which is useful to obtain Gagliardo-Nirenberg type inequalities (see Section \ref{SEC:GN}):
\begin{cor}\label{cor_sem}
Let $F$ and $F^{-1}$ be mappings as in \eqref{sem2} and \eqref{sem3}, respectively. Then we have
\begin{equation} \label{cor_sem_eq1}
Fe^{-tA^{2}}F^{-1}f(r,y)=\frac{1}{\sqrt{4\pi t}}\int_{\mathbb{R}}\exp\left(-\frac{(r-s)^{2}}{4t}\right)f(sy)ds.
\end{equation}
\end{cor}
\begin{proof}[Proof of Corollary \ref{cor_sem}] Plugging $e^{-tA^{2}}=e^{tQ^{2}/4}e^{-t\E^{*}\E}$, \eqref{sem2} and \eqref{sem3} into the left hand side of \eqref{cor_sem_eq1}, and using \eqref{sem1} we obtain
\begin{equation*}
\begin{split}
Fe^{-tA^{2}}F^{-1}f(r,y)&=
F\left(e^{tQ^{2}/4}\frac{e^{-tQ^{2}/4}r^{-Q/2}}{\sqrt{4\pi t}}\int^{\infty}_{0}
e^{-\frac{(\ln r - \ln s )^{2}}{4t}}s^{Q/2-1}(s^{-Q/2}f(\ln s,y))ds\right)\\&=
e^{rQ/2}\frac{e^{-rQ/2}}{\sqrt{4\pi t}}\int^{\infty}_{0}
e^{-\frac{(r - \ln s )^{2}}{4t}}\frac{f(\ln s,y)}{s}ds\\&
=\frac{1}{\sqrt{4\pi t}}\int^{\infty}_{-\infty}
e^{-\frac{(r - s_{1})^{2}}{4t}}f(s_{1}y)ds_{1},
\end{split}
\end{equation*}
which is \eqref{cor_sem_eq1}, where we have used the change of variables $s=e^{s_{1}}$ in the last line.
\end{proof}

\section{Hardy-Sobolev and Gagliardo-Nirenberg type inequalities}
\label{SEC:GN}
In this section we establish a class of the Hardy-Sobolev and Gagliardo-Nirenberg type inequalities on homogeneous groups. Moreover, sharp remainder terms of the Sobolev type inequality are also obtained.

We define the Besov type space $B^{\alpha}(\mathbb{R}\times\wp)$ as the space of all tempered distributions $f$ on $\mathbb{R}\times\wp$ with the norm
\begin{equation}\label{Bes_norm}
\|g\|_{B^{\alpha}(\mathbb{R}\times\wp)}:=\sup_{t>0}\{t^{-\alpha/2}\|F e^{-tA^{2}}F^{-1}f\|_{L^{\infty}(\mathbb{R}\times\wp)}\}<\infty.
\end{equation}
We will also use the one-dimensional case of the following result:
\begin{thm}[{\cite[Theorem 1]{Led03}}]
\label{Led_GN_thm} Let $1\leq p<q<\infty$. Then for every function $f\in L^{p}_{1}(\Rn)$ there exists a positive constant $C=C(p,q,n)$ such that
\begin{equation}\label{Led1}
\|f\|_{L^{q}(\Rn)}\leq C\|\nabla f\|_{L^{p}(\Rn)}^{p/q}\|f\|_{B^{p/(p-q)}_{\infty,\infty}(\Rn)}^{1-p/q},
\end{equation}
where
$$\|f\|_{B^{\alpha}_{\infty,\infty}(\Rn)}:=\sup_{t>0,\;x\in\Rn}\left\{t^{-\alpha/2}
\left|\frac{1}{(4\pi t)^{n/2}}\int_{\Rn}f(y)e^{-|x-y|^{2}/4t}dy\right| \right\}.$$
\end{thm}
Now we state the Gagliardo-Nirenberg type inequalities:
\begin{thm}\label{thm GN} Let $\mathbb{G}$ be a homogeneous group
of homogeneous dimension $Q$. Let $1\leq p<q<\infty$ and let $f$ be such that $\R f\in L^{p}(\mathbb{R}\times\wp)$ and $f\in B^{p/(p-q)}(\mathbb{R}\times\wp)$. Then there exists a positive constant $C=C(p,q)$ such that
\begin{equation} \label{GN1}
\|f \|_{L^{q}(\mathbb{R}\times\wp)}\leq C \|\R f\|^{p/q}_{L^{p}(\mathbb{R}\times\wp)}\|f\|^{1-p/q}_{B^{p/(p-q)}(\mathbb{R}\times\wp)}.
\end{equation}
\end{thm}
\begin{proof}[Proof of Theorem \ref{thm GN}] Using Theorem \ref{Led_GN_thm} with $n=1$ and Corollary \ref{cor_sem}, we obtain
\begin{equation*}
\begin{split}
&\int_{\mathbb{R}}|f(r,y)|^{q}dr\leq C^{q}\int_{\mathbb{R}}\left| \frac{\partial f(r,y)}{\partial r}\right|^{p}dr\\&
\times \left(\sup_{t>0,r \in \mathbb{R}} t^{p/(2(q-p))}\left| \frac{1}{\sqrt{4\pi t}}\int_{\mathbb{R}}e^{-(r-s)^{2}/(4t)}f(s,y)ds\right|\right)^{q-p}\\&
=C^{q}\int_{\mathbb{R}}|\R f(r,y)|^{p}dr
\left(\sup_{t>0,r\in\mathbb{R}}t^{p/(2(q-p))}\left|F e^{-tA^{2}}F^{-1}f(r,y)\right| \right)^{q-p}\\&
\leq C^{q}\int_{\mathbb{R}}|\R f(r,y)|^{p}dr\left(\sup_{t>0}t^{p/(2(q-p))}\left\|F e^{-tA^{2}}F^{-1}f\right\|
_{L^{\infty} (\mathbb{R}\times\wp)} \right)^{q-p}\\&
= C^{q}\int_{\mathbb{R}}|\R f(r,y)|^{p}dr\|f\|^{q-p}_{B^{p/(p-q)}(\mathbb{R}\times\wp)}
\end{split}
\end{equation*}
for any $y\in\wp$, in view of \eqref{Bes_norm}. One obtains \eqref{GN1} after integrating the above inequality with respect to $y$ over $\wp$.
\end{proof}
Once Theorems \ref{sem_thm} and \ref{thm GN}, and Corollary \ref{cor_sem} are established, we obtain the following corollaries in exactly the same way as in \cite[Section 4]{BEHL08}:
\begin{cor}\label{cor1} Let $\mathbb{G}$ be a homogeneous group
of homogeneous dimension $Q$. Let $1\leq p\leq Q-1$ and $\R f \in L^{p}(\mathbb{R}\times \wp)$. Let $p^{*} :=Qp/(Q-p)$. Then we have
\begin{itemize}
\item If $\;\underset{y\in \wp}{\rm sup}\|f(\cdot,y)\|_{L^{p}(\mathbb{R})}<\infty$, then there exists a positive constant $C$ such that
\begin{equation}\label{cor1_eq1}
\|f\|_{L^{p^{*}}(\mathbb{R}\times \wp)}\leq C\|\R f\|^{1/Q}_{L^{p}(\mathbb{R}\times \wp)} \sup_{y \in \wp}\|f(\cdot,y)\|_{L^{p}(\mathbb{R})}^{(Q-1)/Q}.
\end{equation}
\item
If $f\in L^{p}(\mathbb{R}\times \wp)$ and $g(r)=\mathcal{M}(f)(r):=\frac{1}{|\wp|}\int_{\wp}f(r,y)d\sigma(y)$, then there exists a positive constant $C$ such that
\begin{equation}\label{cor1_eq2}
\|g\|_{L^{p^{*}}(\mathbb{R})}\leq C\|\R f\|^{1/Q}_{L^{p}(\mathbb{R}\times \wp)}
\|f\|^{(Q-1)/Q}_{L^{p}(\mathbb{R}\times \wp)}.
\end{equation}
If $f$ is supported in $[-\Lambda,\Lambda]\times \wp$, then there exist positive constants $C_{1}$ and $C_{2}$ such that
\begin{equation}\label{cor1_eq3}
\|f\|_{L^{p^{*}}(\mathbb{R}\times \wp)}\leq C_{1} \Lambda^{(Q-1)/Q^{2}}\|\R f\|^{1/Q}_{L^{p}(\mathbb{R}\times \wp)} \sup_{y \in \wp}\|f(\cdot,y)\|_{L^{p^{*}}(\mathbb{R})}^{(Q-1)/Q}
\end{equation}
and
\begin{equation}\label{cor1_eq4}
\|g\|_{L^{p^{*}}(\mathbb{R})}\leq C_{2}\Lambda^{(Q-1)/Q}\|\R f\|_{L^{p}(\mathbb{R}\times \wp)}.
\end{equation}
\end{itemize}
\end{cor}
\begin{cor}\label{cor2} Let $\mathbb{G}$ be a homogeneous group
of homogeneous dimension $Q$. Let $1\leq p<q<\infty$ and $\R f \in L^{p}(\mathbb{R}\times \wp)$. Then we have
\begin{itemize}
\item If $\;\underset{y\in \wp}{\rm sup}\|f(\cdot,y)\|_{L^{(q/p)-1}(\mathbb{R})}<\infty$, then there exists a positive constant $C$ such that
\begin{equation}\label{cor2_eq1}
\|f\|_{L^{q}(\mathbb{R}\times \wp)}\leq C\|\R f\|^{p/q}_{L^{p}(\mathbb{R}\times \wp)} \sup_{y \in \wp}\|f(\cdot,y)\|_{L^{(q/p)-1}(\mathbb{R})}^{1-p/q}.
\end{equation}
\item
If $f\in L^{(q/p)-1}(\mathbb{R}\times\wp)$ and $g:=\mathcal{M}(f)$, then there exists a positive constant $C$ such that
\begin{equation}\label{cor2_eq2}
\|g\|_{L^{q}(\mathbb{R})}\leq C\|\R f\|^{p/q}_{L^{p}(\mathbb{R}\times \wp)}\|f\|^{1-p/q}_{L^{(q/p)-1}
(\mathbb{R}\times \wp)}.
\end{equation}
\end{itemize}
\end{cor}
\begin{cor}\label{cor3} Let $\mathbb{G}$ be a homogeneous group
of homogeneous dimension $Q\geq 3$. Let $|\cdot|$ be a homogeneous quasi-norm. Let $\E h\in L^{2}(\G)$ and $2^{*}=2Q/(Q-2)$. Let $d\mu=r^{Q-1}dr$. Then we have
\begin{itemize}
\item If $\;\underset{y\in \wp}{\rm sup}\|h(\cdot, y)\|_{L^{2}(\mathbb{R}^{+};d\mu)}<\infty$, then there exists a positive constant $C$ such that
\begin{equation} \label{cor3_eq1}
\begin{split}
\||x|h\|^{2}_{L^{2^{*}}(\G)} &\leq C \left( \|\E h\|^{2}_{L^{2}(\G)}
-\frac{Q^{2}}{4}\|h\|^{2}_{L^{2}(\G)}\right)^{1/Q}\\& \times \sup_{y\in \wp}
\|h(\cdot,y)\|^{2(1-1/Q)}_{L^{2}(\mathbb{R}^{+};d\mu)}.
\end{split}
\end{equation}
\item If $h\in L^{2}(\G)$ and $g:=\mathcal{M}(h)$, then there exists a positive constant $C$ such that
\begin{equation} \label{cor3_eq2}
\||x|g\|^{2}_{L^{2^{*}}(\G)} \leq C \left( \|\E h\|^{2}_{L^{2}(\G)}-\frac{Q^{2}}{4}\|h\|^{2}_{L^{2}(\G)}\right)^{1/Q} \|h\|^{2(1-1/Q)}_{L^{2}(\G)}.
\end{equation}
For $0\leq\delta<Q^{2}/4$, we have
\begin{equation} \label{cor3_eq3}
\||x|h\|^{2}_{L^{2^{*}}(\G)} \leq C (Q^{2}/4-\delta)^{-(Q-1)/Q}\left(\|\E h\|^{2}_{L^{2}(\G)}-\delta\|h\|^{2}_{L^{2}(\G)}\right).
\end{equation}
\end{itemize}
\end{cor}
\begin{cor} \label{cor4} Let $\mathbb{G}$ be a homogeneous group
of homogeneous dimension $Q\geq 3$. Let $|\cdot|$ be a homogeneous quasi-norm. Let $f\in C_{0}^{\infty}(\G\backslash\{0\})$ and $d\mu=r^{Q-1}dr$. Then we have
\begin{itemize}
\item If $\;\underset{y\in \wp}{\rm sup}\|f(\cdot,y)/ |\cdot|\|^{2}_{L^{2}(\mathbb{R}^{+};d\mu)}<\infty$, then there exists a positive constant $C$ such that
\begin{equation}\label{cor4_eq1}
\begin{split}
\|f\|^{2}_{L^{2^{*}}(\G)}&\leq C\left(\|\R f\|^{2}_{L^{2}(\G)}-\left(\frac{Q-2}{2}\right)^{2}\left\|\frac{f}{|x|}\right\|^{2}_{L^{2}(\G)}\right)
^{1/Q}\\&\times \sup_{y \in \wp}\left\{\left\|\frac{f(|\cdot|,y)}{|\cdot|}\right\|^{2}_{L^{2}(\mathbb{R}^{+};d\mu)}\right\}^{1-1/Q}.
\end{split}
\end{equation}
\item If $g :=\mathcal{M}(f)$, then there exists a positive constant $C$ such that
\begin{equation}\label{cor4_eq2}
\|g\|^{2}_{L^{2^{*}}(\G)}\leq C\left(\|\R f\|^{2}_{L^{2}(\G)}-\left(\frac{Q-2}{2}\right)^{2}\left\|\frac{f}{|x|}\right\|^{2}_{L^{2}(\G)}\right)^{1/Q}
\left\|\frac{f}{|x|}\right\|^{(2Q-2)/Q}_{L^{2}(\G)}.
\end{equation}
For $0 \leq\delta <(Q-2)^{2}/4$, we have
\begin{equation}\label{cor4_eq3}
\|g\|^{2}_{L^{2^{*}}(\G)}\leq C\left(\frac{(Q-2)^{2}}{4}-\delta\right)^{-(Q-1)/Q}\left(\|\R f\|^{2}_{L^{2}(\G)}-\delta
\left\|\frac{f}{|x|}\right\|^{2}_{L^{2}(\G)}\right).
\end{equation}
\end{itemize}
\end{cor}
\begin{rem}\label{cor4_rem} In the abelian case $\mathbb{G}=(\Rn,+)$ and $Q=n$, replacing $|\R f|$ by $|\nabla f|$ in Corollary \ref{cor4}, we see that this corollary implies \cite[Corollary 4.5]{BEHL08}, since $|\R f|=|\partial_{r}f|\leq|\nabla f|$.
\end{rem}
\begin{proof}[Proof of Corollary \ref{cor4}] If $h\in C_{0}^{\infty}(\G\backslash\{0\})$, then using the integration by parts and polar coordinates on $\G$ we obtain
\begin{equation}\label{cor4_eq4}
\begin{split}
\int_{\G}|\R(|x|h)|^{2}dx&=\int_{\G}|h+|x|\R h|^{2}dx\\&=
\int_{\G}|h|^{2}dx+\int_{\G}|\E h|^{2}dx+ 2 {\rm Re}\int_{0}^{\infty}\int_{\wp}
h(ry)\overline{\frac{d}{dr}h(ry)}r^{Q}d\sigma(y)dr\\&=
\int_{\G}|h|^{2}dx+\int_{\G}|\E h|^{2}dx+\int_{0}^{\infty}\int_{\wp}
\frac{d}{dr}(|h|^{2})r^{Q}d\sigma(y)dr\\&
=\int_{\G}|h|^{2}dx+\int_{\G}|\E h|^{2}dx- Q\int_{0}^{\infty}\int_{\wp}
|h|^{2}r^{Q-1}d\sigma(y)dr\\&
=\int_{\G}|\E h|^{2}dx-(Q-1)\int_{\G}|h|^{2}dx.
\end{split}
\end{equation}
If we put \eqref{cor4_eq4} and $h=f/|x|$ in \eqref{cor3_eq1}, \eqref{cor3_eq2} and \eqref{cor3_eq3}, then they imply \eqref{cor4_eq1}, \eqref{cor4_eq2} and \eqref{cor4_eq3}, respectively.
\end{proof}
\begin{cor}\label{cor5} Let $\mathbb{G}$ be a homogeneous group
of homogeneous dimension $Q$. Let $|\cdot|$ be a homogeneous quasi-norm.
Let $A_{R}:=\{x\in \G: 1/R\leq |x|\leq R\}$. Let $f\in C^{\infty}_{0}(A_{R})$ and $g :=\mathcal{M}(h)$. Then there exists a positive constant $C$ such that
\begin{equation} \label{cor5_eq1}
\||x|g\|^{2}_{L^{2^{*}}(\G)} \leq C(\ln(R))^{\frac{2(Q-1)}{Q}}
\left(\|\E h\|^{2}_{L^{2}(\G)}-\frac{Q^{2}}{4}\|h\|^{2}_{L^{2}(\G)}\right).
\end{equation}
\end{cor}
\begin{cor}\label{cor51} Let $\mathbb{G}$ be a homogeneous group
of homogeneous dimension $Q$. Let $|\cdot|$ be a homogeneous quasi-norm.
Let $A_{R}:=\{x\in \G: 1/R\leq |x|\leq R\}$. Let $f\in C^{\infty}_{0}(A_{R})$ and $g :=\mathcal{M}(f)$. Then there exists a positive constant $C$ such that
\begin{equation} \label{cor51_eq1}
\|g\|^{2}_{L^{2^{*}}(\G)} \leq C(\ln(R))^{\frac{2(Q-1)}{Q}}
\left(\|\R f\|^{2}_{L^{2}(\G)}-\frac{(Q-2)^{2}}{4}\left\|\frac{f}{|x|}\right\|^{2}_{L^{2}(\G)}\right).
\end{equation}
\end{cor}
\begin{rem}\label{cor51_rem} In the abelian case $\mathbb{G}=(\Rn,+)$ and $Q=n$, as in Remark \ref{cor4_rem}, replacing $|\R f|$ by $|\nabla f|$ in \eqref{cor51_eq1}, we note that this corollary implies \cite[Corollary 4.7]{BEHL08}.
\end{rem}
\begin{proof}[Proof of Corollary \ref{cor51}] By substituting \eqref{cor4_eq4} and $h=f/|x|$ in \eqref{cor5_eq1}, we obtain \eqref{cor51_eq1}.
\end{proof}
\begin{cor}\label{cor6} Let $\mathbb{G}$ be a homogeneous group
of homogeneous dimension $Q$. Let $|\cdot|$ be a homogeneous quasi-norm. Let $2<q<\infty$. Let $f, \E f \in L^{2}(\G)$ and
$$\left\|\frac{f}{|x|^{\frac{Q(6-q)}{2(q-2)}}}\right\|_{L^{q/2-1}(\G)}<\infty.$$ Then there exists a positive constant $C$ such that
\begin{equation*}
\int_{0}^{\infty}|g(r)|^{q}r^{(\frac{Qq}{2}-1)}dr\leq C\left( \|\E f\|^{2}_{L^{2}(\G)}-\frac{Q^{2}}{4}
\|f\|^{2}_{L^{2}(\G)}\right)^{2}\left\|\frac{f}{|x|^{\frac{Q(6-q)}{2(q-2)}}}\right\|^{q-2}_{L^{q/2-1}(\G)},
\end{equation*}
where $$g(|x|)=\mathcal{M}(f)(|x|):=\frac{1}{|\wp|}\int_{\wp}f(|x|,y)d\sigma(y).$$
\end{cor}
Now we prove \eqref{cor3_eq3} with sharp constant, which can be viewed as an analogue of Stubbe's inequality \cite{Stu90} on homogeneous groups.
\begin{thm} \label{Stu_thm} Let $\mathbb{G}$ be a homogeneous group
of homogeneous dimension $Q\geq3$. Let $|\cdot|$ be a homogeneous quasi-norm. Then we have for all $f\in C^{\infty}_{0}(\G)$ and $0\leq \delta <\frac{Q^{2}}{4}$, the inequality
\begin{equation}\label{Stu1}
 \int_{\G}|\E f(x)|^{2}dx-\delta\int_{\G} |f(x)|^{2}dx\geq\frac{\left(\frac{Q^{2}}{4}-\delta\right)^{\frac{Q-1}{Q}}}{\left(\frac{(Q-2)^{2}}{4}\right)^{\frac{Q-1}{Q}}}S_{Q}
\left(\int_{\G}|x|^{2^{*}}|g(|x|)|^{2^{*}}dx\right)^{\frac{2}{2^{*}}}
 \end{equation}
with sharp constant, where $$g(|x|)=\mathcal{M}(f)(|x|):=\frac{1}{|\wp|}\int_{\wp}f(|x|,y)d\sigma(y)$$ and
\begin{equation}\label{S_Q}
S_{Q}:=
|\wp|^{\frac{2}{Q}}Q^{\frac{Q-2}{Q}}(Q-2)\left(\frac{\Gamma(Q/2)\Gamma(1+Q/2)}
{\Gamma(Q)}\right)^{\frac{2}{Q}}.
\end{equation}
 \end{thm}
 \begin{rem}\label{Stu_thm_rem}
In the abelian case $\mathbb{G}=(\Rn,+)$ and $Q=n$, the inequality \eqref{Stu1} gives that
\begin{equation}\label{Stu1_rem}
 \int_{\Rn}|(x\cdot\nabla)f(x)|^{2}dx-\delta\int_{\Rn} |f(x)|^{2}dx\geq\frac{\left(\frac{n^{2}}{4}-\delta\right)^{\frac{n-1}{n}}}{\left(\frac{(n-2)^{2}}{4}\right)^{\frac{n-1}{n}}}S_{n}
\left(\int_{\G}|x|^{2^{*}}|g(|x|)|^{2^{*}}dx\right)^{\frac{2}{2^{*}}}.
 \end{equation}
An interesting observation is that the constant in the above inequality is sharp for any quasi-norm $|\cdot|$, that is, it does not depend on the quasi-norm $|\cdot|$. Therefore, this inequality is new already in the Euclidean setting of $\Rn$. When $|x|=\sqrt{x_{1}^{2}+x_{2}^{2}+...+x_{n}^{2}}$ is the Euclidean distance, the inequality \eqref{Stu1_rem} was investigated in $\Rn$ in \cite[Corollary 4.4]{BEHL08} and in \cite[Theorem 1.1]{Xia11}.
\end{rem}
Before starting the proof of this result, let us recall the following result \cite{Bli30}:
\begin{lem}[{\cite{Bli30}}]\label{Bli_lem}
Let $f$ be a non-negative function. Then for $s\geq0$ and $q>p>1$ we have
\begin{equation}\label{Bli1}
\left(\int^{\infty}_{0}\left|\int^{s}_{0}f(r)dr\right|^{q}r^{q/p-q-1}dr\right)^{p/q}\leq C_{p,q}\int^{\infty}_{0}|f(r)|^{p}dr,
\end{equation}
where
$$C_{p,q}=(q-q/p)^{-p/q}\left(\frac{(q/p-1)\Gamma\left(\frac{pq}{q-p}\right)}
{\Gamma\left(\frac{p}{q-p}\right)\Gamma\left(\frac{p(q-1)}{q-p}\right)}\right)^{(q-p)/q}$$
is sharp. Moreover, the equality in \eqref{Bli1} is attained for functions of the form
\begin{equation}\label{Bli2}
f(r)=c_{1}(c_{2}r^{q/p-1}+1)^{\frac{q}{p-q}}, \;c_{1}>0,\;c_{2}>0.
\end{equation}
\end{lem}
\begin{proof}[Proof of Theorem \ref{Stu_thm}] Let us first prove this theorem for radial functions $f(x)=\widetilde{f}(|x|)$. Then we have $g(r) = \widetilde{f}(r)$ since $g(|x|)=\mathcal{M}(f)(|x|):=\frac{1}{|\wp|}\int_{\wp}f(|x|,y)d\sigma(y)$, and $\E f = |x|\widetilde{f}'(|x|)$. By a direct calculation we have
\begin{equation}\label{Bli3}
\begin{split}
\int_{\G}|\E \widetilde{f}(|x|)|^{2}dx&-\beta(Q-\beta)\int_{\G}|\widetilde{f}(|x|)|^{2}dx\\&
=|\wp|\left(\int^{\infty}_{0}|\widetilde{f}^{\prime}(r)|^{2}r^{Q+1}dr-\beta(Q-\beta)\int^{\infty}_{0}|\widetilde{f}(r)|^{2}r^{Q-1}dr\right),
\end{split}
\end{equation}
where $0\leq\beta<Q/2$. If we set $h(|x|):=|x|^{\beta}\widetilde{f}(|x|)$, then the integration by parts gives that
 \begin{equation}\label{Bli4}
 \begin{split}
 \int^{\infty}_{0}&|h'(r)|^{2}r^{Q+1-2\beta}dr\\&= \int^{\infty}_{0}|\beta r^{\beta-1}\widetilde{f}(r)+r^{\beta}\widetilde{f}'(r)|^{2}r^{Q+1-2\beta}dr\\&
 =\beta^{2}\int^{\infty}_{0}|\widetilde{f}(r)|^{2}r^{Q-1}dr
+\int^{\infty}_{0}|\widetilde{f}'(r)|^{2}r^{Q+1}dr
+\beta\int^{\infty}_{0}\frac{d}{dr}|\widetilde{f}(r)|^{2}r^{Q}dr\\& =\int^{\infty}_{0}|\widetilde{f}'(r)|^{2}r^{Q+1}dr-\beta(Q-\beta)\int^{\infty}_{0}|\widetilde{f}(r)|^{2}r^{Q-1}dr.
 \end{split}
 \end{equation}
 On the other hand, we have by changing the variables $s=r^{Q-2\beta}$ that
\begin{equation}\label{Bli5}
\begin{split}
\int^{\infty}_{0}|h'(r)|^{2}r^{Q+1-2\beta}dr&=
\int^{\infty}_{0}|(Q-2\beta)s^{\frac{Q-2\beta-1}{Q-2\beta}}h'(s)|^{2}s^{\frac{Q-2\beta+1}{Q-2\beta}}\frac{s^{\frac{1}{Q-2\beta}-1}ds}{Q-2\beta}\\&
=(Q-2\beta)
\int^{\infty}_{0}s^{2}|h'(s)|^{2}ds.
\end{split}
\end{equation}
Taking into account \eqref{Bli4} and \eqref{Bli5}, we rewrite \eqref{Bli3} as
\begin{equation}\label{Bli6}
\begin{split}
\int_{\G}|\E \widetilde{f}(|x|)|^{2}dx&-\beta(Q-\beta)\int_{\G}|\widetilde{f}(|x|)|^{2}dx\\&
=|\wp|(Q-2\beta)\left(\int^{\infty}_{0}s^{2}|h'(s)|^{2}ds\right).
\end{split}
\end{equation}
Now denoting $\phi(s):=h'(s)$ and $\psi(s):=s^{-2}\phi(s^{-1})$, and using Lemma \ref{Bli_lem} with $p=2$ and $q=2^{*}$ we have
\begin{equation*}
\begin{split}\int^{\infty}_{0}&s^{2}|h'(s)|^{2}ds
=\int^{\infty}_{0}s^{2}|\phi(s)|^{2}ds=\int^{\infty}_{0}|\psi(s)|^{2}ds\\&\geq
\left(\frac{Q}{Q-2}\right)^
{\frac{Q-2}{Q}}\left(\frac{\Gamma(Q/2)\Gamma(1+Q/2)}
{\Gamma(Q)}\right)^{\frac{2}{Q}}\left(\int^{\infty}_{0}\left|
\int^{s}_{0}|\psi(t)|dt\right|^{2^{*}}s^{\frac{2-2Q}{Q-2}}ds\right
)^{\frac{2}{2^{*}}}\\&=
\left(\frac{Q}{Q-2}\right)^
{\frac{Q-2}{Q}}\left(\frac{\Gamma(Q/2)\Gamma(1+Q/2)}
{\Gamma(Q)}\right)^{\frac{2}{Q}}\left(\int^{\infty}_{0}\left|
\int^{\infty}_{s^{-1}}|\phi(t)|dt\right|^{2^{*}}s^{\frac{2-2Q}{Q-2}}ds\right
)^{\frac{2}{2^{*}}}\\&\geq
\left(\frac{Q}{Q-2}\right)^
{\frac{Q-2}{Q}}\left(\frac{\Gamma(Q/2)\Gamma(1+Q/2)}
{\Gamma(Q)}\right)^{\frac{2}{Q}}\left(\int^{\infty}_{0}|
h(s^{-1})|^{2^{*}}s^{\frac{2-2Q}{Q-2}}ds\right
)^{\frac{2}{2^{*}}}\\&=
\left(\frac{Q}{Q-2}\right)^
{\frac{Q-2}{Q}}\left(\frac{\Gamma(Q/2)\Gamma(1+Q/2)}
{\Gamma(Q)}\right)^{\frac{2}{Q}}\left(\int^{\infty}_{0}|
h(s)|^{2^{*}}s^{\frac{2}{Q-2}}ds\right
)^{\frac{2}{2^{*}}}\\&=
(Q-2\beta)^{\frac{2}{2^{*}}}\left(\frac{Q}{Q-2}\right)^
{\frac{Q-2}{Q}}\left(\frac{\Gamma(Q/2)\Gamma(1+Q/2)}
{\Gamma(Q)}\right)^{\frac{2}{Q}}\left(\int^{\infty}_{0}|
r\widetilde{f}(r)|^{2^{*}}r^{Q-1}dr\right
)^{\frac{2}{2^{*}}},
\end{split}
\end{equation*}
where we have used $s=r^{Q-2\beta}$ and $h(r)=r^{\beta}\widetilde{f}(r)$ in the last line. Combining this with \eqref{Bli6}, we arrive at
\begin{equation}\label{Bli7}
\begin{split}
\int_{\G}&|\E \widetilde{f}(|x|)|^{2}dx-\beta(Q-\beta)\int_{\G}|\widetilde{f}(|x|)|^{2}dx\\&
\geq|\wp|(Q-2\beta)^{\frac{2Q-2}{Q}}\left(\frac{Q}{Q-2}\right)^
{\frac{Q-2}{Q}}\left(\frac{\Gamma(Q/2)\Gamma(1+Q/2)}
{\Gamma(Q)}\right)^{\frac{2}{Q}}\left(\int^{\infty}_{0}|
r\widetilde{f}(r)|^{2^{*}}r^{Q-1}dr\right
)^{\frac{2}{2^{*}}}\\&
=|\wp|^{\frac{2}{Q}}(Q-2\beta)^{\frac{2Q-2}{Q}}\left(\frac{Q}{Q-2}\right)^
{\frac{Q-2}{Q}}\left(\frac{\Gamma(Q/2)\Gamma(1+Q/2)}
{\Gamma(Q)}\right)^{\frac{2}{Q}}\left(\int_{\G}|x|^{2^{*}}|
\widetilde{f}(|x|)|^{2^{*}}dx\right
)^{\frac{2}{2^{*}}}.
\end{split}
\end{equation}
Here, if we set $\beta=(Q-\sqrt{Q^{2}-4\delta})/2$ for $0\leq \delta<Q^{2}/4$, then recalling $g(|x|)=\widetilde{f}(|x|)$ we see that \eqref{Bli7} implies that
\begin{equation}\label{Bli71}
\begin{split}
\int_{\G}&|\E \widetilde{f}(|x|)|^{2}dx-\delta\int_{\G}|\widetilde{f}(|x|)|^{2}dx\\&\geq
|\wp|^{\frac{2}{Q}}(Q^{2}-4\delta)^{\frac{Q-1}{Q}}\left(\frac{Q}{Q-2}\right)^
{\frac{Q-2}{Q}}\left(\frac{\Gamma(Q/2)\Gamma(1+Q/2)}
{\Gamma(Q)}\right)^{\frac{2}{Q}}\left(\int_{\G}|x|^{2^{*}}|
g(|x|)|^{2^{*}}dx\right
)^{\frac{2}{2^{*}}}\\&=
\left(\frac{Q^{2}-4\delta}{(Q-2)^{2}}\right)^{\frac{Q-1}{Q}}S_{Q}\left(\int_{\G}|x|^{2^{*}}|
g(|x|)|^{2^{*}}dx\right
)^{\frac{2}{2^{*}}},
\end{split}
\end{equation}
since \eqref{S_Q}. Thus, we have obtained \eqref{Stu1} with sharp constant for all radial functions $f\in C^{\infty}_{0}(\G)$.

Finally, using Lemma \ref{Euler_lem_proper2}, and taking into account $g(|x|)=\widetilde{f}(|x|)$ in \eqref{Bli71}, we obtain \eqref{Stu1} for non-radial functions. The constant in \eqref{Stu1} is sharp, since this constant is sharp for radial functions by Lemma \ref{Bli_lem}.
\end{proof}
\section{Maximal Hardy inequality}
\label{SEC:Max}
In this section we discuss a weighted exponential inequality.
\begin{thm}\label{weight2_hardy_thm}
Let $\phi$ and $\psi$ be positive functions defined on $\G$. Then there exists a positive constant $C$ such that
\begin{equation}\label{high2_wei_Hardy1}
\int_{\G}\phi(x)\exp (\mathcal{M} \log f)(x)dx\leq C\int_{\G} \psi(x)f(x)dx
\end{equation}
holds for all positive $f$ if and only if
\begin{equation}\label{high2_wei_Hardy2}
A:=\sup_{R>0} R^{Q}\int_{|x|\geq R}\frac{\phi(x)\exp\left(\mathcal{M}\log\frac{1}{\psi}\right)(x)}{|x|^{2Q}}dx<\infty,
\end{equation}
where $(\mathcal{M}f)(x)=\frac{1}{|B(0,|x|)|}\int_{B(0,|x|)}f(z)dz$.
\end{thm}
\begin{rem}\label{weight2_hardy_thm_rem} In the abelian case $\mathbb{G}=(\Rn,+)$ and $Q=n$, the inequality \eqref{high2_wei_Hardy1} was studied in \cite{HKK01} for $n=1$, and in \cite{DHK97} for $n\geq1$.
\end{rem}
\begin{proof}[Proof of Theorem \ref{weight2_hardy_thm}] First, we show \eqref{high2_wei_Hardy2}$\Rightarrow$\eqref{high2_wei_Hardy1}. Denoting
\begin{equation}\label{W3}
W_{3}(x):=\phi(x)\exp\left(\mathcal{M}\log\frac{1}{\psi}\right)(x)
\end{equation}
and $u(x):=f(x)\psi(x)$, and changing the variables $z=|x|\xi$, we obtain
\begin{equation}\label{high2_wei_Hardy3}
\begin{split}
\int_{\G}&\phi(x)\exp (\mathcal{M} \log f)(x)dx\\&=
\int_{\G}\phi(x)\exp\left(\frac{1}{|B(0,|x|)|}\left(\int_{B(0,|x|)}\log\left(\frac{1}{\phi}\right)(z)dz+
\int_{|z|\leq |x|}\log (\phi f)(z)dz\right)\right)dx\\&=
\int_{\G}\phi(x)\exp\left(\mathcal{M}\log\frac{1}{\phi}\right)(x)
\exp\left(\frac{1}{|B(0,|x|)|}\int_{|z|\leq |x|}\log(\phi(z)f(z))dz\right)dx\\&
=\int_{\G}W_{3}(x)\exp\left(\frac{1}{|B(0,|x|)|}\int_{|z|\leq|x|}\log (u(z))dz\right)dx\\&
=\int_{\G}W_{3}(x)\exp\left(\frac{1}{|x|^{Q}|B(0,1)|}\int_{B(0,1)}\log (u(|x|\xi))|x|^{Q}d\xi\right)dx.
\end{split}
\end{equation}
Now taking into account
$\int_{B(0,1)}\log(|\xi|^{Q})d\xi=Q\int_{\wp}\int_{0}^{1}r^{Q-1}\log rdrd\sigma(y)
=-|B(0,1)|
$
and using Jensen's inequality,  we obtain from \eqref{high2_wei_Hardy3} that
\begin{equation}\label{high2_wei_Hardy4}
\begin{split}
&\int_{\G}\phi(x)\exp (\mathcal{M} \log f)(x)dx\\&=
\int_{\G}W_{3}(x)\exp \left(\frac{1}{|B(0,1)|}\left(\int_{B(0,1)}\log(|\xi|^{Q}u(|x|\xi))d\xi-
\int_{B(0,1)}\log(|\xi|^{Q})d\xi\right)\right)dx\\&=\int_{\G}W_{3}(x)\exp \left(\frac{1}{|B(0,1)|}\int_{B(0,1)}\log(|\xi|^{Q}u(|x|\xi))d\xi+1\right)dx\\&
=e\int_{\G}W_{3}(x)\exp \left(\frac{1}{|B(0,1)|}\int_{B(0,1)}\log(|\xi|^{Q}u(|x|\xi))d\xi\right)dx\\&
\leq \frac{e}{|B(0,1)|}\int_{\G}W_{3}(x)\int_{B(0,1)}|\xi|^{Q}u(|x|\xi)d\xi dx\\&=
\frac{e}{|B(0,1)|}\int_{\wp}\int_{\wp}\int_{0}^{\infty}r^{Q-1}W_{3}(rw)\int_{0}^{1}s^{2Q-1}u(rsy)dsdrd\sigma(y) d\sigma(w),
\end{split}
\end{equation}
where $|\xi|=s$ and $|x|=r$. Here, we continue our calculation by changing the variables $rs=t$ to get
 \begin{equation}\label{high2_wei_Hardy5}
\begin{split}
&\int_{\G}\phi(x)\exp (\mathcal{M} \log f)(x)dx\\&
\leq \frac{e}{|B(0,1)|}\int_{\wp}\int_{\wp}\int_{0}^{\infty}r^{Q-1}W_{3}(rw)\int_{0}^{1}s^{2Q-1}u(rsy)dsdrd\sigma(y) d\sigma(w)\\&
=\frac{e}{|B(0,1)|}\int_{\wp}\int_{\wp}\int_{0}^{1}s^{2Q-1}
\int_{0}^{\infty}W_{3}\left(\frac{t}{s}w\right)\left(\frac{t}{s}\right)^{Q-1}u(ty)\frac{dt}{s}dsd\sigma(y) d\sigma(w)\\&=
\frac{e}{|B(0,1)|}\int_{\wp}\int_{\wp}\int_{0}^{1}s^{Q-1}
\int_{0}^{\infty}W_{3}\left(\frac{t}{s}w\right)t^{Q-1}u(ty)dtdsd\sigma(y) d\sigma(w)\\&=
\frac{e}{|B(0,1)|}\int_{\wp}\int_{\wp}\int_{0}^{\infty}t^{Q-1}u(ty)\left(\int_{0}^{1}
s^{Q-1}W_{3}\left(\frac{tw}{s}\right)ds\right)dt d\sigma(y) d\sigma(w)\\&=
\frac{e}{|B(0,1)|}\int_{\wp}\int_{\wp}\int_{0}^{\infty}t^{Q-1}u(ty)\left(\int_{t}^{\infty}
\left(\frac{t}{r}\right)^{Q-1}W_{3}(rw)t\frac{dr}{r^{2}}\right)dt d\sigma(y) d\sigma(w)\\&
=\frac{e}{|B(0,1)|}\int_{\wp}\int_{\wp}\int_{0}^{\infty}t^{2Q-1}u(ty)\left(\int_{t}^{\infty}
r^{-Q-1}W_{3}(rw)dr\right)dt d\sigma(y) d\sigma(w)\\&\leq
\frac{e}{|B(0,1)|}\int_{\wp}\int_{\wp}\int_{0}^{\infty}t^{Q-1}u(ty)t^{Q}
\left(\int_{t}^{\infty}\frac{r^{Q-1}W_{3}(rw)dr}{r^{2Q}}\right)dtd\sigma(w)d\sigma(y)\\&
=\frac{e}{|B(0,1)|}\int_{\G}u(x)\left(|x|^{Q}\int_{|z|\geq |x|}\frac{W_{3}(z)}{|z|^{2Q}}dz\right)dx\\&
\leq A \frac{e}{|B(0,1)|}\int_{\G}u(x)dx,
\end{split}
\end{equation}
yielding \eqref{high2_wei_Hardy1}, where we have used \eqref{high2_wei_Hardy2} in the last line.

Now let us show \eqref{high2_wei_Hardy1}$\Rightarrow$\eqref{high2_wei_Hardy2}. We note, using \eqref{high2_wei_Hardy3}, that \eqref{high2_wei_Hardy1} is equivalent to
\begin{equation}\label{high2_wei_Hardy6}
\int_{\G}W_{3}(x)\exp\left(\frac{1}{|B(0,|x|)|}\int_{|z|\leq|x|}\log (u(z))dz\right)dx\leq C\int_{\G}u(x)dx.
\end{equation}
Here, choosing the following radial function for a fixed $R>0$,
\begin{equation}\label{u}
u(x)=R^{-Q}\chi_{(0,R)}(|x|)+e^{-2Q}|x|^{-2Q}R^{Q}\chi_{(R,\infty)}(|x|),\;\;x\in\G,
\end{equation}
we compute
\begin{equation*}
\begin{split}
\int_{\G}&W_{3}(x)\exp\left(\frac{1}{|B(0,|x|)|}\int_{|z|\leq|x|}\log (u(z))dz\right)dx\\& \leq C\int_{\G}u(x)dx\\&
=C\int_{\wp}\int_{0}^{\infty}s^{Q-1}u(s)dsd\sigma(y)\\&
=C|\wp|\left(\int_{0}^{R}s^{Q-1}R^{-Q}ds+\int_{R}^{\infty}e^{-2Q}s^{Q-1}R^{Q}s^{-2Q}ds\right)\\&
=C|\wp|\left(\frac{1}{Q}+\frac{e^{-2Q}}{Q}\right)=:C(Q)<\infty,
\end{split}
\end{equation*}
since $\chi$ is the cut-off function.
Now, taking into account this and plugging \eqref{u} into the left hand side of \eqref{high2_wei_Hardy6} we get
\begin{equation*}
\begin{split}
\infty&>C(Q)\geq\int_{\G}W_{3}(x)\exp\left(\frac{1}{|B(0,|x|)|}\int_{|z|\leq|x|}\log (u(z))dz\right)dx\\&
=\int_{\wp}\int_{0}^{\infty}s^{Q-1}W_{3}(sy)\exp\left(\frac{1}{|B(0,s)|}
\int_{\wp}\int_{0}^{s}r^{Q-1}\log (u(r))drd\sigma(w)\right)dsd\sigma(y)\\&
=\int_{\wp}\int_{0}^{\infty}s^{Q-1}W_{3}(sy)\exp \left(\frac{|\wp|}{s^{Q}|B(0,1)|}\int_{0}^{s}r^{Q-1}\log (u(r))dr\right)dsd\sigma(y)\\&
=\int_{\wp}\left(\int_{0}^{R}s^{Q-1}W_{3}(sy)\exp \left(\frac{|\wp|}{s^{Q}|B(0,1)|}\int_{0}^{s}r^{Q-1}\log(u(r))dr\right)ds\right)d\sigma(y)
\end{split}
\end{equation*}
\begin{equation*}
\begin{split}
&+\int_{\wp}\left(\int_{R}^{\infty}s^{Q-1}W_{3}(sy)\exp\left(\frac{|\wp|}{s^{Q}|B(0,1)|}\left(\int_{0}^{R}r^{Q-1}\log (R^{-Q})dr\right.\right.\right.\\&
\left.\left.\left. +\int_{R}^{s}r^{Q-1}\log(e^{-2Q}r^{-2Q}R^{Q})dr\right)\right)ds\right)d\sigma(y)\\&
\geq \int_{\wp}\left(\int_{R}^{\infty}s^{Q-1}W_{3}(sy)\exp\left(\frac{|\wp|}{s^{Q}|B(0,1)|}\left(\int_{0}^{R}r^{Q-1}\log (R^{-Q})dr\right.\right.\right.\\&
\left.\left.\left.+\int_{R}^{s}r^{Q-1}\log(e^{-2Q}r^{-2Q}R^{Q})dr\right)\right)ds\right)d\sigma(y)
\\&=\int_{\wp}\left(\int_{R}^{\infty}s^{Q-1}W_{3}(sy)\exp\left(\frac{|\wp|}{s^{Q}|B(0,1)|}\left(\int_{0}^{R}r^{Q-1}\log (R^{-Q})dr\right.\right.\right.\\&
\left.\left.\left.+\int_{R}^{s}r^{Q-1}\log (e^{-2Q})dr-2Q \int_{R}^{s}r^{Q-1}(\log r) dr+\int_{R}^{s}r^{Q-1}\log (R^{Q})dr\right)\right)ds\right)d\sigma(y)\\&
=\int_{\wp}\left(\int_{R}^{\infty}s^{Q-1}W_{3}(sy)\exp\left(\frac{|\wp|}{s^{Q}|B(0,1)|}\left(\frac{R^{Q}\log (R^{-Q})}{Q}-2Q\frac{s^{Q}-R^{Q}}{Q}\right.\right.\right.\\&
\left.\left.\left.-2Q\left[\frac{r^{Q}\log r}{Q}-\frac{r^{Q}}{Q^{2}}\right]^{s}_{R}
+\frac{s^{Q}-R^{Q}}{Q}\log (R^{Q})\right)\right)ds\right)d\sigma(y)
\end{split}
\end{equation*}
\begin{equation*}
\begin{split}
&=\int_{\wp}\int_{R}^{\infty}s^{Q-1}W_{3}(sy)\exp\left(\frac{|\wp|}{|B(0,1)|}\left(-2+\frac{2R^{Q}}{s^{Q}}-2\log s+\frac{2}{Q}-\frac{2R^{Q}}{Qs^{Q}}+\log R\right)\right)dsd\sigma(y)\\&
\geq e^{(2-2Q)}\int_{\wp}\int_{R}^{\infty}s^{Q-1}W_{3}(sy)\frac{R^{\frac{|\wp|}{|B(0,1)|}}}{s^{\frac{2|\wp|}{|B(0,1)|}}}dsd\sigma(y)\\&
=e^{2-2Q}R^{Q}\int_{|x|\geq R}\frac{W_{3}(x)}{|x|^{2Q}}dx\\&
=e^{2-2Q}R^{Q}\int_{|x|\geq R}\frac{\phi(x)\exp\left(\mathcal{M}\log\frac{1}{\psi}\right)(x)}{|x|^{2Q}}dx,
\end{split}
\end{equation*}
which implies \eqref{high2_wei_Hardy2}, where we have used $\frac{|\wp|}{|B(0,1)|}=Q$, $\frac{2R^{Q}}{s^{Q}}-\frac{2R^{Q}}{Qs^{Q}}>0$ and \eqref{W3} in the last two lines.
\end{proof}

\section{Further inequalities}
\label{final}
In this section we discuss a number of related inequalities, also interesting on their own.
\begin{thm}\label{thm-radial-weight}
	For any quasi-norm $|\cdot|$, all differentiable $|\cdot|$-radial functions $\phi$, all $p>1$, $Q\geq 2$, and all $f \in C_{0}^{1}(\mathbb{G})$ we have
	\begin{equation}\label{thm_first}
	\int_{\mathbb{G}} \frac{\phi^{\prime}(|x|)}{|x|^{Q-1}}|f(x)|^{p}dx \leq \int_{\mathbb{G}} \left| \R f(x) \right|^{p} dx  + (p-1)\int_{\mathbb{G}} \frac{|\phi(|x|)|^{\frac{p}{p-1}}}{|x|^{\frac{p(Q-1)}{p-1}}}  |f(x)|^{p}dx,
	\end{equation}	
	and
		\begin{equation}\label{thm_second}
	\int_{\mathbb{G}} \frac{\phi^{\prime}(|x|)}{|x|^{Q-1}}|f(x)|^{p}dx \leq p \left(\int_{\mathbb{G}} \left| \R f(x) \right|^{p} dx \right)^{\frac{1}{p}}  \left(\int_{\mathbb{G}} \frac{|\phi(|x|)|^{\frac{p}{p-1}}}{|x|^{\frac{p(Q-1)}{p-1}}}  |f(x)|^{p}dx\right)^{\frac{p-1}{p}}.
	\end{equation}	
\end{thm}

In \eqref{thm_first} taking $\phi=\log |x|$  in the Euclidean (Abelian) case ${\mathbb G}=(\mathbb R^{n},+)$, $n\geq 2$, we have
$Q=n$, and taking $p=n\geq 2$, so for any quasi-norm $|\cdot|$ on $\mathbb R^{n}$ it implies
the new inequality:
\begin{align}
\int_{\mathbb R^{n}}
\frac{|f|^{n}}{|x|^{n}} dx \leq \int_{\mathbb R^{n}} \left| \R f \right|^{n} dx  + (n-1)\int_{\mathbb R^{n}} \frac{\left|\log \frac{1}{|x|}\right|^{\frac{n}{n-1}}}{|x|^{n}} \left| f \right|^{n} dx, \nonumber
\end{align}	
which in turn, by using Schwarz's inequality with the standard Euclidean distance $\|x\|=\sqrt{x^{2}_{1}+\ldots+x^{2}_{n}}$, implies the `critical' Hardy inequality
\begin{equation}\label{critical_H}
\int_{\mathbb R^{n}}
\frac{|f|^{n}}{\|x\|^{n}} dx \leq \int_{\mathbb R^{n}} \left| \nabla f \right|^{n} dx  + (n-1)\int_{\mathbb R^{n}} \frac{\left|\log \frac{1}{\|x\|}\right|^{\frac{n}{n-1}}}{\|x\|^{n}} \left| f \right|^{n} dx,
\end{equation}	
where $\nabla $ is the standard gradient on $\mathbb R^{n}$.
It is known (see, e.g. \cite[Section 1.2.5]{BEL15}) that there is
no positive constant $C$ such that
\begin{align}
\int_{\mathbb R^{n}}
\frac{|f|^{n}}{\|x\|^{n}} dx \leq C \int_{\mathbb R^{n}} \left| \nabla f \right|^{n} dx \nonumber
\end{align}	
for all $f\in C^{1}_{0}(\mathbb R^{n}).$
Therefore, the appearance of a positive additional
term (the second term) on the right hand side of \eqref{critical_H} seems essential. We refer to \cite{Ruzhansky-Suragan:critical} and references therein for different versions of critical Hardy-Sobolev type inequalities. Note that this type of inequalities (Hardy-Sobolev type inequalities with an additional term on the right hand side) can be applied, for example in the Euclidean case, to establish the existence and nonexistence of positive exponentially bounded weak solutions to a parabolic type operator perturbed by a critical singular potential (see \cite{ST18}).

In \eqref{thm_second} taking $\phi=|x|^{n}$ in the Euclidean (Abelian) case ${\mathbb G}=(\mathbb R^{n},+)$, $n\geq 2$, we have
$Q=n$, so for any quasi-norm $|\cdot|$ on $\mathbb R^{n}$ it implies
the following uncertainty principle:
	\begin{equation}
\int_{\mathbb{R}^{n}}|f|^{p}dx \leq  \frac{p}{n} \left(\int_{\mathbb{R}^{n}} \left| \R f \right|^{p} dx \right)^{\frac{1}{p}}  \left(\int_{\mathbb{R}^{n}} |x|^{\frac{p}{p-1}} |f|^{p}dx\right)^{\frac{p-1}{p}},
\end{equation}	
which in turn, by using Schwarz's inequality with the standard Euclidean distance $\|x\|=\sqrt{x^{2}_{1}+\ldots+x^{2}_{n}}$, implies that
\begin{equation}\label{Heisenberg-Pauli-Weyl uncertainty principle}
\int_{\mathbb{G}}|f|^{p}dx \leq  \frac{p}{n} \left(\int_{\mathbb R^{n}} \left| \nabla f \right|^{p} dx \right)^{\frac{1}{p}}  \left(\int_{\mathbb R^{n}} \|x\|^{\frac{p}{p-1}} |f|^{p}dx\right)^{\frac{p-1}{p}},
\end{equation}	
where $\nabla $ is the standard gradient on $\mathbb R^{n}$. In the case when $p=2$ we have
\begin{equation}\label{HPWp=2}
\left(\int_{\mathbb R^{n}}|f|^{2}dx\right)^{2} \leq  \left(\frac{2}{n}\right)^{2}  \int_{\mathbb R^{n}} \left| \nabla f \right|^{2} dx  \int_{\mathbb R^{n}} \|x\|^{2} |f|^{2}dx, \quad n\geq2,
\end{equation}
for all $f\in C^{1}_{0}(\mathbb R^{n}).$
The same inequality with the constant $\left(\frac{2}{n-2}\right)^{2}, n\geq3,$ (instead of $\left(\frac{2}{n}\right)^{2}$) is known as the Heisenberg-Pauli-Weyl uncertainty principle (see, e.g. \cite[Remark 2.10]{Ruzhansky-Suragan:uncertainty}), that is,
\begin{equation}\label{RS17}
\left(\int_{\mathbb R^{n}}|f|^{2}dx\right)^{2} \leq  \left(\frac{2}{n-2}\right)^{2}  \int_{\mathbb R^{n}} \left| \nabla f \right|^{2} dx  \int_{\mathbb R^{n}} \|x\|^{2} |f|^{2}dx, \quad n\geq3,
\end{equation}
for all $f\in C^{1}_{0}(\mathbb R^{n}).$
Thus, when $n=2$ inequality \eqref{HPWp=2} gives the critical case of the Heisenberg-Pauli-Weyl uncertainty principle. Moreover, since $\frac{2}{n-2}\geq \frac{2}{n},\, n\geq3,$ inequality \eqref{HPWp=2} is an improved version of \eqref{RS17}. Note that equality case in \eqref{HPWp=2} holds for the function $f=C\exp (-b \|x\|),\,b>0.$

\begin{proof}[Proof of Theorem \ref{thm-radial-weight}]
	By applying the polarization formula and integration by parts we obtain
	\begin{align*}
	&	
	\int_{\mathbb{G}}\frac{\phi^{\prime}(|x|)}{|x|^{Q-1}}|f|^{p}dx
	= \int_{0}^{\infty}\int_{\wp} |f|^{p} \frac{\phi^{\prime}(r)}{r^{Q-1}} r^{Q-1}d\sigma(y)dr\\
	& =  \int_{0}^{\infty}\int_{\wp} |f|^{p} \frac{d}{d r} \phi (r)  d\sigma(y)dr	=-\int_{0}^{\infty}\int_{\wp}  \phi(r) \frac{d}{d r} |f|^{p} d\sigma(y)dr
	\\
	& =  - \int_{\mathbb{G}} \frac{\phi(|x|)}{|x|^{Q-1}}\R |f|^{p} dx
	=- p {\rm Re} \int_{\mathbb{G}} \frac{\phi(|x|)|f|^{p-2}f}{|x|^{Q-1}} \overline{\R f} dx.
	\end{align*}
	Furthermore, using the Young inequality for $p>1$, we arrive at	
	\begin{align} \label{themainstep}
	&	
	\int_{\mathbb{G}}\frac{\phi^{\prime}(|x|)}{|x|^{Q-1}}|f|^{p}dx
	= - p {\rm Re} \int_{\mathbb{G}} \frac{\phi|f|^{p-2}f}{|x|^{Q-1}} \overline{\R f} dx \nonumber
	\\ &
	\leq p \int_{\mathbb{G}} \frac{|\phi||f|^{p-1}}{|x|^{Q-1}} |\R f| dx \\
	& \leq \int_{\mathbb{G}} \left| \R f \right|^{p} dx + (p-1)\int_{\mathbb{G}} \frac{|\phi(|x|)|^{\frac{p}{p-1}}}{|x|^{\frac{p(Q-1)}{p-1}}}  |f|^{p}dx. \nonumber
	\end{align}
	This proves inequality \eqref{thm_first}.
	On the other hand, from \eqref{themainstep} using the H\"{o}lder inequality for $p>1$, we establish
	
		\begin{align*}
	&	
	\int_{\mathbb{G}}\frac{\phi^{\prime}(|x|)}{|x|^{Q-1}}|f|^{p}dx
	\leq p \int_{\mathbb{G}} \frac{|\phi||f|^{p-1}}{|x|^{Q-1}} |\R f| dx \\
	& \leq p \left(\int_{\mathbb{G}} \left| \R f \right|^{p} dx \right)^{\frac{1}{p}}  \left(\int_{\mathbb{G}} \frac{|\phi(|x|)|^{\frac{p}{p-1}}}{|x|^{\frac{p(Q-1)}{p-1}}}  |f|^{p}dx\right)^{\frac{p-1}{p}}.
	\end{align*}
	This completes the proof.
\end{proof}
Note that a similar proof technique gives
\begin{thm}\label{thm-radial-weight-pq}
	For any quasi-norm $|\cdot|$, all differentiable $|\cdot|$-radial functions $\phi$, all $p>1$, $Q\geq 2$ with $\frac{1}{p}+\frac{1}{q}=1$, and all $f \in C_{0}^{1}(\mathbb{G})$ we have
	\begin{equation}\label{thm_first-pq}
	\int_{\mathbb{G}} \frac{\phi^{\prime}(|x|)}{|x|^{Q-1}}|f|^{p}dx \leq \int_{\mathbb{G}} \left| \R f \right|^{p} dx  + \frac{p}{q}\int_{\mathbb{G}} \frac{|\phi(|x|)|^{q}}{|x|^{q(Q-1)}} |f|^{p}dx,
	\end{equation}	
	and
	\begin{equation}\label{thm_second-pq}
	\int_{\mathbb{G}} \frac{\phi^{\prime}(|x|)}{|x|^{Q-1}}|f|^{p}dx \leq p \left(\int_{\mathbb{G}} \left| \R f \right|^{p} dx \right)^{\frac{1}{p}}  \left(\int_{\mathbb{G}} \frac{|\phi(|x|)|^{q}}{|x|^{q(Q-1)}}  |f|^{p}dx\right)^{\frac{1}{q}}.
	\end{equation}	
\end{thm}

Usually, classical Hardy and Sobolev type inequalities are
stated with a gradient. As a reader noticed througout this paper by using the Euler or/and Radial operators we have obtained different type of gradient free functional inequalities. On the one hand, such analysis is important since, as we mentioned in the introduction, in general there is no homogeneous gradient on homogeneous (Lie) groups.
On the other hand, these inequalities give new inequalities even in Euclidean cases as well as cover classical inequalities with gradients. However, in addition to these methods there are other techniques to obtain gradient free Hardy-Sobolev type inequalities.
To conclude these discussions let us introduce the following
functional
$$I_{\delta}(f):=\int\int_{|f(y)-f(x)|>\delta}\frac{\delta^{2}}{|y^{-1}\circ x|^{Q+2}}dxdy,$$
where $\circ$ is the group operation on $\mathbb{G}$.

In the Euclidean (Abelian) case by using this functional, H.-M. Nguyen and M. Squassina (see, e.g. \cite{Nguyen08} and \cite{NS17}) obtained nonlocal verions of the classical Hardy-Sobolev type inequality. Note that from their inequalities in the singular limit $\delta \searrow 0$ one recovers the classical Hardy-Sobolev type results since in the Euclidean case
the functional $I_{\delta}$ converges to the Dirichlet energy up to a normalisation constant. Their main inequalities are gradient free ones. Therefore, those are extendable to the homogeneous (Lie) groups. We believe that such ideas of proofs of (nonlocal) gradient free inequalities can be generalised to the homogeneous groups. Below we demonstrate this idea in a special case.

\begin{prop}\label{prop1}
	Let $
	Q\geq 3$, $I_{\delta}(f)<\infty$ and
	\begin{equation}\label{1.1}
	\int_{\{|f|>\lambda_{Q} \delta\}} |f|^{\frac{2Q}{Q-2}} dx \leq C_{Q} I_{\delta} (f)^{\frac{Q}{Q-2}},\quad \delta >0, \end{equation}
	for some positive constants $C_{Q}$ and $\lambda_{Q}$.
	Then for all $\delta >0$  there exists a positive constant $C_{Q}$ such that
	\begin{equation*}
	\int_{\mathbb{G}} \frac{|f|^2}{\| f\|^2_{L^2 (\mathbb{G})}} \log \frac{|f|^2}{\| f\|^2_{L^2 (\mathbb{G})}} dx +\frac{Q}{2}\log \| f\|^2_{L^2 (\mathbb{G})} \leq \frac{Q}{2} \log \left( C_{Q} \delta^{\frac{4}{Q}} \| f\|^{\frac{2Q-4}{Q}}_{L^2 (\mathbb{G})} + C_{Q} I_{\delta}(f)\right).
	\end{equation*}
\end{prop}
\begin{proof}[Proof of Proposition \ref{prop1}]
	The proof of Proposition \ref{prop1} relies on the same technique as the proof of  \cite[Theorem 1.1]{NS17} with
	the difference that now the quasi-norm (instead of the Euclidean distance) is used. For the proof we only need to recall the fact that  the Lebesque measure on $\mathbb R^{n}$ gives the Haar measure for $\mathbb{G}$.
	The rest of the proof is exactly the same as in the proof of \cite[Theorem 1.1]{NS17}.
\end{proof}

\end{document}